\documentclass[11pt]{article} 

\usepackage{etoolbox}
\newtoggle{siam}

\togglefalse{siam} 

\iftoggle{siam}{

\usepackage{epsf}
\usepackage{epsfig}
\usepackage{enumerate}
\usepackage{url}

\usepackage{amsmath,amssymb}
\usepackage{wasysym}
\newtheorem{proposition*}{Proposition} 
\newtheorem{definition*}{Definition} 
\newtheorem{remark*}{Remark} 
\newtheorem{remark}[theorem]{Remark} 

\def\cM{{\cal M}}
\def\cO{{\cal O}}
\def\cS{{\cal S}}
\def\cT{{\cal T}}
\def\Z{{\rm {{\rm Z}\kern-.28em{\rm Z}}}}
\def\R{{\rm \hbox{I\kern-.2em\hbox{R}}}}
\def\be{\begin{equation}}
\def\ee{\end{equation}}
\def\<{\langle}
\def\>{\rangle}
\def\sm{\setminus}
\def\gsim{\hbox{\kern -.2em\raisebox{-1ex}{$~\stackrel{\textstyle>}{\sim}~$}}\kern -.2em}
\def\lsim{\hbox{\kern -.2em\raisebox{-1ex}{$~\stackrel{\textstyle<}{\sim}~$}}\kern -.2em}
\DeclareMathOperator\interp{{\rm I}}
\newtheorem{remarkIntro}{Remark}
\def\vp{\varphi}
\newcommand\trans{\mathrm T}
\DeclareMathOperator\argmin{argmin}
\newcommand\stext[1]{\ \text{ #1 } \ }
\def\ve{\varepsilon}

\newcommand\tr{\textcolor{red}}
\DeclareMathOperator\Id{Id}

}{

\usepackage{epsf}
\usepackage{epsfig}
\usepackage{enumerate}
\usepackage{url}

\voffset=-2cm
\hoffset=-1.5cm
\textwidth=16cm
\textheight=23cm

\usepackage{amsmath,amsfonts,amssymb,amsthm}
\usepackage{wasysym}
\usepackage[retainorgcmds]{IEEEtrantools}

\usepackage{hyperref}

\def\vp{\varphi}
\def\<{\langle}
\def\>{\rangle}

\def\sm{\setminus}

\def\cB{{\cal B}}

\def\cM{{\cal M}}

\def\cO{{\cal O}}

\def\cS{{\cal S}}
\def\cT{{\cal T}}

\def\Chi{\raise .3ex
\hbox{\large $\chi$}} \def\vp{\varphi}
\def\lsima{\hbox{\kern -.6em\raisebox{-1ex}{$~\stackrel{\textstyle<}{\sim}~$}}\kern -.4em}
\def\lsim{\hbox{\kern -.2em\raisebox{-1ex}{$~\stackrel{\textstyle<}{\sim}~$}}\kern -.2em}
\def\({\Bigl (}
\def\){\Bigr )}
\def\({\Bigl (}
\def\){\Bigr )}

\newcommand{\be}{\begin{equation}}
\newcommand{\ee}{\end{equation}}
\newcommand{\bea}{$$ \begin{array}{lll}}
\newcommand{\eea}{\end{array} $$}
\newcommand{\bi}{\begin{itemize}}
\newcommand{\ei}{\end{itemize}}

\newtheorem{theorem}{Theorem}[section]
\newtheorem{remark}[theorem]{Remark}
\newtheorem{lemma}[theorem]{Lemma}
\newtheorem{proposition}[theorem]{Proposition}
\newtheorem{corollary}[theorem]{Corollary}
\newtheorem{definition}[theorem]{Definition}

\newtheorem*{remark*}{Remark}
\newtheorem*{proposition*}{Proposition}
\newtheorem*{definition*}{Definition}

\def\I{{\rm \hbox{I\kern-.2em\hbox{I}}}}
\def\P{{\rm \hbox{I\kern-.2em\hbox{P}}}}
\def\sP{{\rm \hbox{\scriptsize I\kern-.2em\hbox{\scriptsize P}}}}
\def\H{{\rm \hbox{I\kern-.2em\hbox{H}}}}
\def\R{{\rm \hbox{I\kern-.2em\hbox{R}}}}
\def\N{{\rm \hbox{I\kern-.2em\hbox{N}}}}
\def\Z{{\rm {{\rm Z}\kern-.28em{\rm Z}}}}
\def\C{{\rm \hbox{C\kern -.5em {\raise .32ex \hbox{$\scriptscriptstyle
|$}}\kern-.22em{\raise .6ex \hbox{$\scriptscriptstyle |$}}\kern .4em}}}
\def\sC{{\rm \hbox{\scriptsize \rm C\kern -.6em {\raise .4ex \hbox{$\scriptscriptstyle \scriptsize|$}}\kern-.22em{\raise .5ex \hbox{$\scriptscriptstyle \scriptsize |$}}\kern .4em}}}

\def\ve{\varepsilon}

\def\cO{\mathcal O}

\newcommand\trans{\mathrm T}

\def\cM{\mathcal M}

\def\bac{$$\left\{\begin{array}}
\def\eac{\end{array}\right.$$}
\def\beqa{\begin{eqnarray*}}
\def\eeqa{\end{eqnarray*}}

\def\gsim{\hbox{\kern -.2em\raisebox{-1ex}{$~\stackrel{\textstyle>}{\sim}~$}}\kern -.2em}

\DeclareMathOperator\Id{Id}

\DeclareMathOperator\argmin{argmin}

\DeclareMathOperator\supp{supp}

\DeclareMathOperator\interp{{\rm I}}

\newcommand\tr{\textcolor{red}}

\newtheorem{question}{Question}
\newtheorem{programme}[question]{Programmation}

\def\bques{\begin{question}}
\def\eques{\end{question}}
\def\bprog{\begin{programme}}
\def\eprog{\end{programme}}

\usepackage{stmaryrd}

\def\II{\mathrm {I\kern-0.1exI}}

\newcommand\stext[1]{\ \text{ #1 } \ }

\def\sR{{\rm \hbox{\scriptsize I\kern-.2em\hbox{\scriptsize R}}}}

\def\sN{{\rm \hbox{\scriptsize I\kern-.2em\hbox{\scriptsize N}}}}
\def\sC{{\rm \hbox{\scriptsize \rm C\kern -.6em {\raise .4ex \hbox{$\scriptscriptstyle \scriptsize|$}}\kern-.22em{\raise .5ex \hbox{$\scriptscriptstyle \scriptsize |$}}\kern .4em}}}

}

\def\pathPic{Illustrations/}

\usepackage{color}
\usepackage{algorithm}
\usepackage{printlen}

\DeclareMathOperator\length{length}

\DeclareMathOperator\dist{d}
\DeclareMathOperator\distC{D}

\def\trial{trial}
\def\accepted{accepted}
\def\dim{m}
\DeclareMathOperator\distance{dist}

\def\dtimes{\distance}
\def\assign{\leftarrow}
\def\cB{\mathcal B}

\def\Vladimirsky{A.\ Vladimirsky}
\usepackage{multirow}

\begin{document}
\title{Anisotropic Fast-Marching on cartesian grids\\ using Lattice Basis Reduction%
\thanks{
This work was partly supported by ANR grant NS-LBR ANR-13-JS01-0003-01. 
}
} 
\author{Jean-Marie Mirebeau\footnote{CNRS, University Paris Dauphine, UMR 7534, Laboratory CEREMADE, Paris, France.}}
\maketitle
\date{}
\begin{abstract}
We introduce a modification of the Fast Marching algorithm, which solves the anisotropic eikonal equation associated to an arbitrary continuous Riemannian metric $\cM$, on a two or three dimensional domain. 
The algorithm has a complexity $\cO(N \ln N + N \ln \kappa(\cM))$, where $N$ is the discrete domain cardinality. The logarithmic dependency in the maximum anisotropy ratio $\kappa(\cM)$ of the Riemannian metric allows to handle extreme anisotropies for a limited numerical cost. 
We prove the consistence of the algorithm, and illustrate its efficiency by numerical experiments.
The algorithm relies on the computation at each grid point $z$ 
of a special system of coordinates: a \emph{reduced} basis of the lattice $\Z^m$, with respect to the symmetric positive definite matrix $\cM(z)$ encoding the desired anisotropy at this point. 
\end{abstract}

\section*{Introduction}

The anisotropic Eikonal equation, or static Hamilton-Jacobi equation, is a Partial Differential Equation (PDE) which describes an elementary front propagation model: the speed of the front depends only on the front position and orientation. 
This PDE is encountered in numerous applications, such as motion planning control problems \cite{SV03}, modeling of bio-medical phenomena \cite{SKD07}, and image analysis \cite{PPKC10}. 
It was also recently used in the context of medical image analysis \cite{BC10} for extracting vessels in two dimensional projections or three dimensional scans of the human body, and for performing virtual endoscopies. This application requires to solve a highly anisotropic generalized eikonal equation with a high resolution on a cartesian grid, at a computational cost compatible with user interaction. It is one of our key motivations.

This paper is devoted to the construction and the study of a new algorithm, Fast Marching using Lattice Basis Reduction (FM-LBR), designed to solve the anisotropic eikonal equation associated to a given Riemannian metric $\cM$, and able to handle large or even extreme anisotropies. The domain must be of dimension two or three, and discretized on a cartesian grid.
The FM-LBR, as its name indicates, is a variant of the classical Fast Marching algorithm \cite{SV03,T95}, an efficient method for solving the eikonal equation when the metric is isotropic (proportional at each point to the identity matrix).
Lattice Basis Reduction \cite{NS09} is a concept from discrete mathematics, used in the FM-LBR to produce sparse causal stencils for the discretization of the eikonal equation; it allows to benefit in an optimal way of the interplay between the Riemannian geometric structure of the PDE, and the arithmetic structure of the discretization grid.
A similar technique is used in \cite{FM13} to construct sparse non-negative stencils for anisotropic diffusion.

In order to illustrate the specificity of our approach, we need to introduce some notation. 
Denote by $S_m^+$ the collection of $m\times m$ symmetric positive definite matrices, and associate to each $M \in S_m^+$ the norm $\|u\|_M := \sqrt{\<u, M u\>}$ on $\R^m$. Consider a bounded open domain $\Omega \subset \R^m$, equipped with a Riemannian metric $\cM \in C^0(\overline \Omega, S_m^+)$. We address the anisotropic eikonal equation: find the unique viscosity \cite{L82} solution $\distC : \overline \Omega \to \R$ of
\be
\label{eikonal}
\left\{
\begin{array}{rl}
\|\nabla \distC(z)\|_{\cM(z)^{-1}} = 1 & \text{ for almost every } z\in \Omega,\\
\distC = 0 & \text{ on } \ \partial \Omega.
\end{array}
\right.
\ee
See the end of \S \ref{sec:MainResults} for different boundary conditions, and algorithmic restrictions on the dimension $m$.
Introduce the Riemannian length of a Lipschitz path $\gamma : [0,1] \to \overline \Omega$:
\begin{equation}
\label{def:Length}
\length(\gamma) := \int_0^1 \|\gamma'(t) \|_{\cM(\gamma(t))} dt,
\end{equation}
and denote by $\distC(x,y)$ the length of the shorted path joining $x,y \in \overline \Omega$, also referred to as the Riemannian distance between these points. 
The PDE \eqref{eikonal} admits an optimal control interpretation: $\distC(x) = \min \{\distC(x,y); \, y \in \partial \Omega\}$ is the minimal distance from $x \in \Omega$ to the boundary. 

Consider a discrete set $Z \subset \R^m$, which in our case will be a cartesian grid. 
Discretizations of \eqref{eikonal} take the form of a fixed point problem: find $\dist : Z \to \R$ such that 
\be
\label{discreteSys}
\left\{
\begin{array}{ll}
\dist(z) = \Lambda(\dist,z)  & \text{ for all } z\in Z \cap \Omega,\\
\dist(z) = 0 & \text{ for all } z \in Z \sm \Omega.
\end{array}
\right.
\ee
This formulation involves the Hopf-Lax update operator $\Lambda(\dist,z)$ \cite{SV03,BR06,KushnerDupuis92,GonzalesRofman85}, which mimics at the discrete level Belmann's optimality principle associated to the optimal control interpretation of \eqref{eikonal}. See Appendix \ref{subsec:Accuracy} for details on this principle, the approximations underlying its discretization, and their accuracy. The definition of $\Lambda(\dist, x)$ involves a mesh (or \emph{stencil}) $V(x)$ of a small neighborhood of $x \in Z \cap \Omega$, with vertices on $Z$, and reads:
\begin{equation}
\label{def:HopfLax}
\Lambda(\dist,x) := \min_{y \in \partial V(x)} \|x-y\|_{\cM(x)} + \interp_{V(x)} \dist (y),
\end{equation}
where $\interp_{V}$ denotes piecewise linear interpolation on a mesh $V$. (In this paper, a mesh in $\R^m$ is a finite collection $\cT$ of $m$-dimensional non-flat simplices, which is conforming in the sense that the intersection $T \cap T'$ of any $T,T' \in \cT$ is the convex hull of their common vertices.)

Numerical solvers of the eikonal equation differ by (i) the construction of the stencils $V(z)$, $z\in \Omega$, and (ii) the approach used to solve the system \eqref{discreteSys}, which is inspired by the algorithms of Bellmann-Ford or of Dijkstra used for computing distances in graphs, instead of continuous domains. 
The algorithm presented in this paper, FM-LBR, belongs to the category of Dijkstra inspired algorithms with static stencils, and among these is the first one to guarantee a uniform upper bound on the stencil cardinality, independently of the  Riemannian metric $\cM$. The anisotropy ratio $\kappa(M)$ of a  matrix $M\in S_\dim^+$, and the maximum anisotropy $\kappa( \cM)$ of the Riemannian metric $\cM$, are defined as follows:
\be
\label{defKappa}
\kappa(M) := \max_{\|u\|=\|v\|=1} \frac{\|u\|_M}{\|v\|_M} = \sqrt{\|M\| \|M^{-1}\|}, \qquad \kappa(\cM) := \max_{z\in \overline \Omega} \kappa(\cM(z)).
\ee
We denote by $N :=  \#(Z \cap \Omega)$ the cardinality of the discrete domain.
\begin{itemize}
\item
{\bf Bellman-Ford inspired algorithms.}
The discrete fixed point problem \eqref{discreteSys} is solved via Gauss-Seidel iteration: the replacement rule $\dist(z_k) \assign \Lambda(\dist, z_k)$ is applied for $k=0,1,2,...$ to a mutable map $\delta : Z \cap \Omega \to \R$, until a convergence criterion is met. 
In the fast sweeping methods, see \cite{TsaiChengOsherZhao03} and references therein,
the sequence of points $(z_k)_{k \geq 0}$ enumerates repeatedly the lines and the columns of $Z \cap \Omega$. 
Alternatively this sequence is obtained via a priority queue in the Adaptive Gauss-Seidel Iteration (AGSI) of Bornemann and Rasch \cite{BR06}. The stencil $V(z)$ of a point $z\in Z \cap \Omega$ is usually the offset by $z$ of a fixed stencil $V$ given at the origin, such as those illustrated on Figure \ref{fig:Classical}.

Fast sweeping methods have $\cO(\lambda(\cM) N)$ complexity when the metric $\cM$ is \emph{isotropic} (proportional to the identity at each point), but this result does not extend to anisotropic Riemannian metrics, see \cite{Zhao05} for the proof and the expression of $\lambda(\cM)$. The AGSI has complexity $\cO(\mu(\cM) N^{1+\frac 1 \dim})$, for arbitrary anisotropic Riemannian metrics, where $\mu(\cM)$ is a non explicit constant which depends on global geometrical features of the metric \cite{BR06}. The AGSI is a popular, simple, and quite efficient method, which is included for comparison in our numerical tests.

\item 
{\bf Dijkstra inspired algorithms.} The system \eqref{discreteSys} is solved in a single pass, non-iteratively, using an  ordering of $Z \cap \Omega$ determined at run-time. 
This is possible provided the Hopf-Lax update operator satisfies the so-called ``causality property'', see Proposition \ref{prop:Causality}, which can be ensured if the stencil $V(z)$ of each $z\in Z \cap \Omega$ satisfies some geometrical properties depending on $\cM(z)$, see Definition \ref{def:Mesh}. 
The different Dijkstra inspired methods are characterized by the construction of the stencils $V(z)$, in contrast with Bellman-Ford inspired methods which are characterized by the choice of the sequence $(z_k)_{k \geq 0}$.
Solving the system \eqref{discreteSys} with a Dijkstra inspired algorithm has complexity $\cO(\mu(\cM) N \ln N)$, where $\mu(\cM)$ is an upper bound on the cardinality of the stencils (the number of simplices they are built of).

In the Ordered Upwind Method (OUM) of Sethian and Vladimirsky \cite{SV03,VladTime06}, the stencils are constructed at run-time; their cardinality is bounded by $\cO(\kappa(\cM)^\dim)$ and drops to $\cO(\kappa(\cM)^{\dim-1})$ as $N \to \infty$. In contrast, the stencils are constructed during a preprocessing step and then static in the Monotone Acceptance Ordered Upwind Method (MAOUM) of Alton and Mitchell \cite{AltonMitchell12}; their cardinality is bounded by $\cO(\kappa(\cM)^\dim)$. The FM-LBR introduced in the present work uses a similar approach, except that the stencils cardinality is $\cO(1)$, fully independent of the Riemannian metric $\cM$. The complexity estimates are thus $\cO(\kappa(\cM)^\dim N \ln N)$ for the OUM and the MAOUM (asymptotically $\cO(\kappa(\cM)^{\dim-1} N \ln N)$ for the OUM), and $\cO(N \ln N + N \ln \kappa(\cM))$ for our approach the FM-LBR, where the second term in the complexity accounts for the stencil construction. 
\end{itemize}

The above mentioned algorithms are consistent for the anisotropic eikonal equation associated to an arbitrary continuous Riemannian metric $\cM \in C^0(\overline \Omega, S_\dim^+)$, in the sense that the discrete output $\dist_h$ of the algorithm executed on the grid $Z_h := h \Z^2$, of scale $h>0$, converges to the viscosity solution $\distC$ of the continuous problem \eqref{eikonal} as $h \to 0$. Some more specialized variants of the fast marching algorithm are only consistent for a restricted set of metrics, but can be executed nonetheless with an arbitrary anisotropic metric $\cM$; in that case the discrete system \eqref{discreteSys} may not be solved, and the numerical results are variable, see \S \ref{sec:num}.
For instance the original fast marching algorithm \cite{T95} is consistent if $\cM(z)$ is proportional to the identity matrix for each $z\in \Omega$, and more generally if $\cM(z)$ is a diagonal matrix. In addition to these cases of isotropy and axis-aligned anisotropy, some variants are also consistent if the metric anisotropy $\kappa(\cM)$ is smaller than a given bound $\kappa_0$, see  \cite{JBTDPIB08} and Figure \ref{fig:Classical}.
Our numerical experiments \S \ref{sec:num} include for comparison one of these methods: Fast Marching using the 8 point stencil (FM-8, center left stencil on Figure \ref{fig:Classical}), which is popular in applications \cite{BC10} thanks to its short computation times and despite the lack of convergence guarantee for arbitrary metrics.
Depending on the implementation \cite{RS09,SV03,T95}, involving either a sorted list or a bucket sort, these methods have complexity $\cO(N\ln N)$ or $\cO(\Upsilon(\cM) N)$, 
where 
$$
\Upsilon(\cM) := \sqrt{\max_{z\in \Omega} \|\cM(z)\| \max_{z'\in \Omega} \|\cM(z')^{-1}\|}.
$$
In the applications for which our method is intended, one typically has $\ln(N) \lsim \kappa(\cM) \leq \Upsilon( \cM) \ll N$, in such way that the complexity $\cO(N \ln N + N \ln \kappa(\cM))$ of the proposed method is comparable to $\cO(N\ln N)$ and smaller than $\cO(\Upsilon(\cM) N)$.
In summary, the FM-LBR combines the universal consistency (i.e.\ for any Riemannian metric) of the AGSI, OUM and MAOUM, with a quasi-linear complexity just as the original Fast Marching algorithm.

 
\begin{remark*} 
Each solver of the eikonal equation comes with a specific construction of the stencils $V(z)$, $z \in \Omega$. The highly efficient, but specialized, approach used in the FM-LBR limits its potential for generalization, see the end of \S \ref{sec:MainResults}. Static adaptive stencils also have a memory impact, see Remark \ref{rem:memory}.
Since the Hopf-Lax update operator \eqref{def:HopfLax} depends on the stencils, the discrete solution $\dist$ of  \eqref{discreteSys} is scheme dependent, and so is its accuracy.
See Appendix \ref{subsec:Accuracy} for a heuristic accuracy analysis, and \cite{M12b} for the case of constant metrics.
Numerical experiments \S \ref{sec:num}, on application inspired test cases, show that the FM-LBR accuracy is competitive. 
\end{remark*}

\begin{figure}
\begin{center}
\includegraphics[width=3cm]{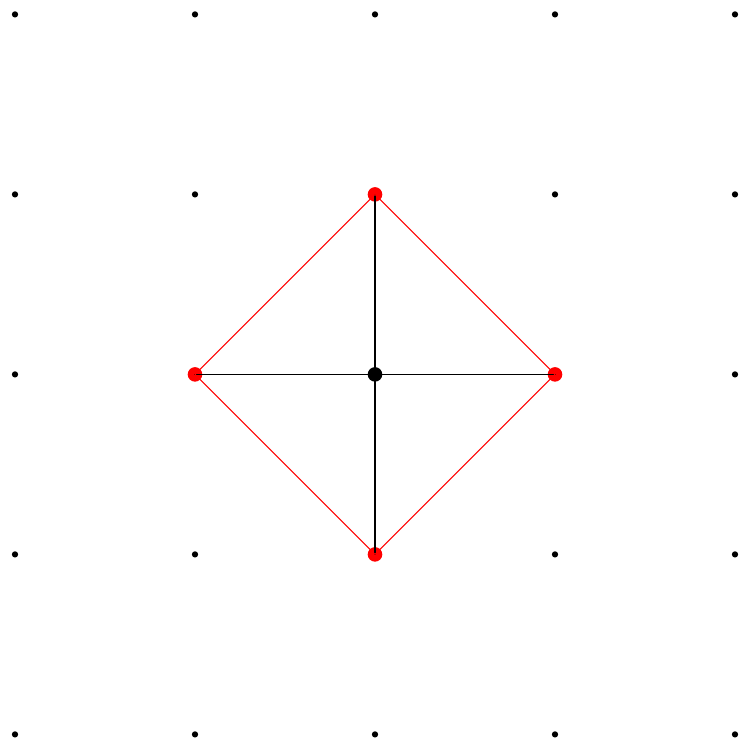}
\includegraphics[width=3cm]{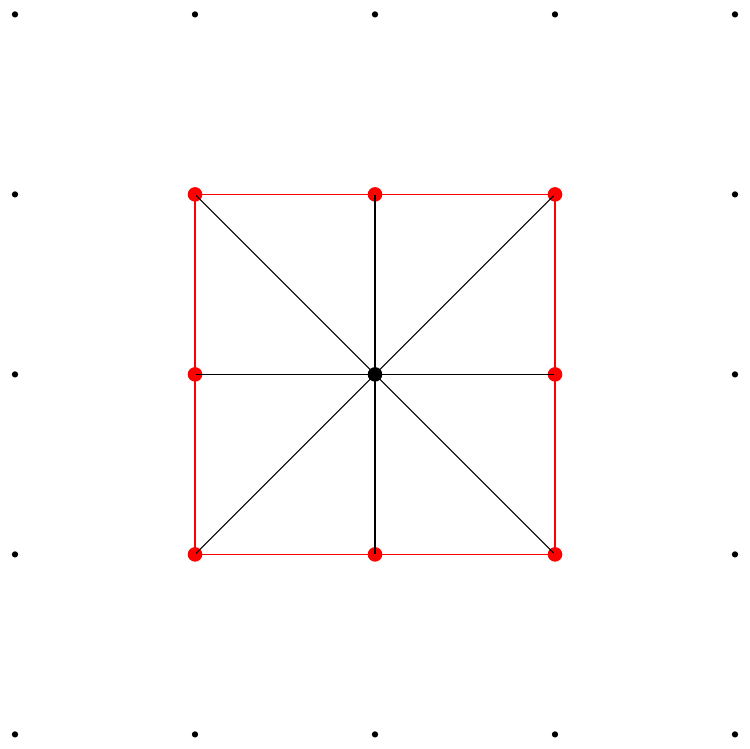}
\includegraphics[width=3cm]{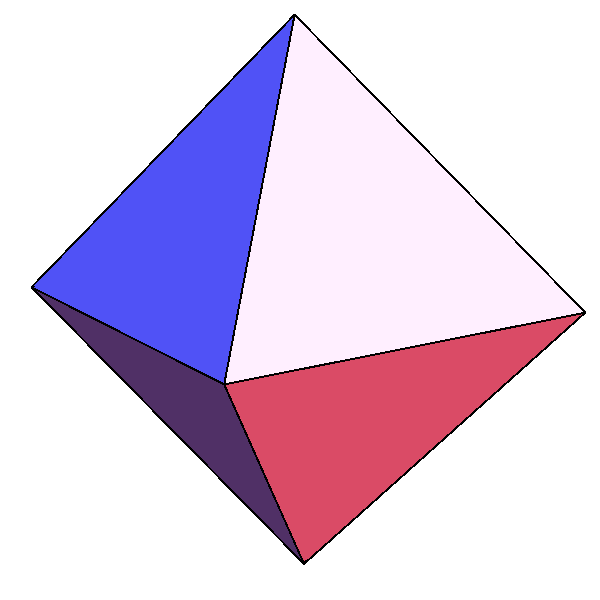}
\includegraphics[width=3cm]{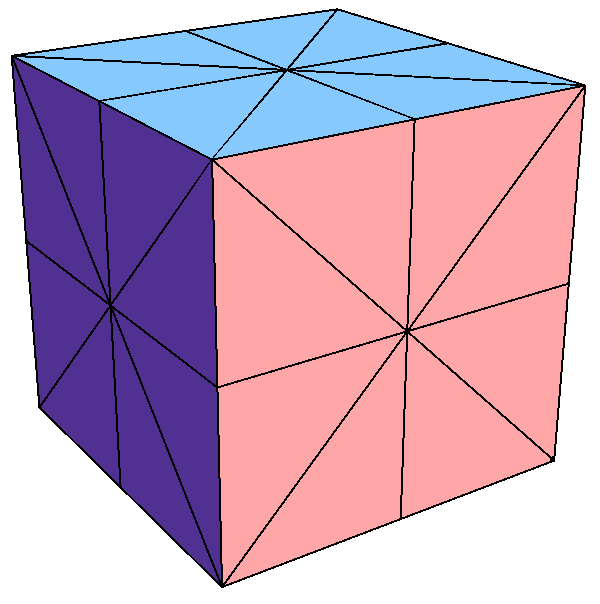}
\end{center}
\caption{
\label{fig:Classical}
Some classical stencils used in the discretization of two dimensional (left) or three dimensional (right) eikonal equations. 
The meshes are $M$-acute, a property implying discrete causality, see Definition \ref{def:Mesh}, for matrices $M$ which are diagonal or of anisotropy ratio $\kappa(M)$ bounded by respectively $1, \ 1+\sqrt 2, \ 1, \ (\sqrt 3 +1)/2$ (from left to right). 
}
\end{figure}

\section{Sparse causal stencils for the anisotropic eikonal equation}
\label{sec:MainResults}

Our main contribution is the construction of discretization stencils for anisotropic eikonal equations, which have a uniformly bounded cardinality and preserve a structural property of the PDE: causality, inherited from its interpretation as a deterministic control problem. Discrete causality, allowing to solve the fixed point system \eqref{discreteSys} in a single pass, is the following property.

\begin{proposition}[Causality property, J.A.\ Sethian, A.\ Vladimirsky, Appendix of \cite{SV03}] 
\label{prop:Causality}
Let $x \in \R^m$, and let $V$ be a finite mesh of a neighborhood of $x$. Let $M \in S_m^+$ and let us assume that $\<y-x, M (z-x) \> \geq 0$ for any vertices $y$, $z$ of a common face of $\partial V$ (acuteness condition). 
Consider a discrete map $\dist$ defined on the vertices of $V$, and the optimization problem
\begin{equation*}
\Lambda := \min_{y \in \partial V} \|x-y\|_M + \interp_V \dist (y),
\end{equation*}
where $\interp$ denotes piecewise linear interpolation. Then the minimum defining $\Lambda$ is attained on a $k$-face $[y_1, \cdots, y_k]$ of $\partial V$, with $k \leq m$, such that  $\Lambda > \dist(y_i)$ for all $1 \leq i \leq k$.
\end{proposition}
A mesh $V$ satisfying the geometric acuteness condition of Proposition \ref{prop:Causality}, is called a causal stencil, at the point $x$, and with respect to $M$.

Consider the directed graph $G$, associated to the fixed point problem \eqref{discreteSys}, where for all $x,y \in Z$ we place an arrow $x \to y$ iff the minimum defining $\dist(x) = \Lambda(\dist,x)$ is attained on a face of the stencil $\partial V(x)$ containing $y$. The causality property, states that the presence of an arrow $x \to y$ implies $\dist(x) > \dist(y)$; in particular the graph $G$ has no cycles, hence the system of equations \eqref{discreteSys} does not feature any dependency loop. The Fast-Marching algorithm \cite{T95} traces back theses dependencies, determining at run-time an  ordering $(x_i)_{i=1}^N$, of the discrete domain $Z \cap \Omega$, such that the distances $(\dist(x_i))_{i=1}^N$ are increasing. We reproduce this method for completeness, see Algorithm \ref{algo:FastMarching}, but refer to its original introduction \cite{T95}, or its use with alternative static adaptive stencils in the MAOUM \cite{AltonMitchell08}, for the proof that it solves the discrete system \eqref{discreteSys}. 

\begin{algorithm}
\label{algo:FastMarching}
\caption{The Fast Marching algorithm, with static stencils adapted to a metric.} 
\begin{tabular}{l}
\textbf{Input}:  The values $\cM(z)$ of a Riemannian metric, for all $z \in Z \cap \Omega$.\\
\textbf{Construct}  causal stencils $V(z)$, with respect to  $\cM(z)$, at all points $z\in Z \cap \Omega$.\\
\textbf{Construct}  the reversed stencils, defined by $V[y] := \{ z \in Z\cap \Omega ; \, y \text{ is a vertex of } V(z)\}$.\\
\textbf{Initialize}  a (mutable) table $\dist : Z \to \R$, to $+ \infty$ on $Z \cap \Omega$, and to $0$ elsewhere.\\
\textbf{Initialize}  a (mutable) boolean table $b : Z \to \{\trial, \accepted\}$ with $b(y) \assign \trial$ iff $V[y]\neq \emptyset$.\\
\textbf{While} there remains a $\trial$ point (i.e.\ $y\in Z$ such that $b(y) = \trial$) \textbf{do}\\
\phantom{Whi} Denote by $y$ a $\trial$ point which minimizes $\dist$, and set $b(y) \assign \accepted$.\\
\phantom{Whi} For all $x \in V[y]$, set $\dist(x) \assign \min \{ \dist(x), \, \Lambda(\dist,x; \, b,y)\}$.\\
\textbf{Output}:  The map $\dist : Z \to \R$.
\end{tabular}
\end{algorithm}

We denoted by $\Lambda(\dist,x; \, b, y)$ the modification of the Hopf-Lax update operator \eqref{def:HopfLax} in which the minimum is only taken over faces (of any dimension) of $\partial V(x)$ which vertices (i) contain $y$, and (ii) are all $\accepted$.
Regarding the FM-LBR complexity $\cO(N \ln N + N \ln \kappa(\cM))$, we refer for details to the classical analysis in \cite{T95,SV03,AltonMitchell08} and simply point out that (i) each FM-LBR causal stencil costs $\cO(\ln \kappa(\cM))$ to construct, (ii) maintaining a list of $\Omega \cap Z$, sorted by increasing values of the mutable map $\dist$, costs $\cO(\ln N)$ for each modification of a single value of $\dist$, with a proper heap sort implementation, and (iii) the optimization problem defining the Hopf-Lax update \eqref{def:HopfLax}, or its variant $\Lambda(\dist,x; \, b, y)$, has an explicit solution: the minimum associated to each face of $\partial V(x)$ is the root of a simple univariate quadratic polynomial, see  Appendix of \cite{SV03}. Memory usage is discussed in detail in Remark \ref{rem:memory}.

As announced, we limit our attention to PDE discretizations on cartesian grids, of the form
\begin{equation}
\label{def:Grid}
Z = h R (\xi + \Z^m),
\end{equation}
where $h>0$ is a scaling parameter, $R$ a rotation, and $\xi \in \R^m$ an offset (in \eqref{def:Grid} and \eqref{eq:VFromT} we abusively apply geometric transformations not only to points, but also to sets of points and meshes). The use of an unbounded grid \eqref{def:Grid} is a mathematical artifact aimed to simplify the exposition; only points of $Z$ close to $\Omega$ (precisely: such that $V[y] \neq \emptyset$) play an active role in Algorithm \ref{algo:FastMarching}. We may construct a causal stencil at $x \in Z$ with respect to $M\in S_m^+$, by suitably scaling, translating and rotating an $R^\trans M R$-acute mesh $\cT$, defined below:
\begin{equation}
\label{eq:VFromT}
V = x + h \,R \,\cT. 
\end{equation}
\begin{definition}
\label{def:Mesh}
An $M$-acute mesh, where $M \in S_m^+$, is an $m$-dimensional mesh $\cT$ covering a neighborhood of the origin, and such that the vertices $v_0, \cdots , v_m$ of any simplex $T \in \cT$ satisfy (i) $v_0=0$, (ii) $(v_1, \cdots, v_m)$ form a basis of $\Z^m$, and (iii) $\<v_i, M v_j\>\geq 0$ for all $1 \leq i < j \leq m$.
\end{definition}
The condition $|\det(v_1, \cdots, v_m)|=1$ 
ensures that $0$ (resp.\ $x$) is the only grid point in the interior of the domain covered by $\cT$ (resp.\ $V$), hence that information does not ``fly over'' some grid points when solving \eqref{discreteSys}. 
Applying the next proposition to the classical stencils of Figure \ref{fig:Classical}, we obtain that they are causal for Riemannian metrics of limited anisotropy.
\begin{proposition}
\label{propGammaT}
Let $\cT$ be an $\dim$-dimensional mesh satisfying the requirements of Definition \ref{def:Mesh}, except (iii) (which does not make sense without a given matrix $M \in S_m^+$). Let 
\be
\label{defKappaT}
\kappa(\cT) := \sqrt{\frac{1+\gamma(\cT)}{1-\gamma(\cT)}}, \quad \text{where} \quad \gamma(\cT) := \min_{T,(u,v)} \frac{\<u,v\>}{\|u\| \|v\|},
\ee
and where the minimum in $\gamma(\cT)$ is taken among all non-zero vertices $u,v$ of a common simplex $T\in \cT$. 
The mesh $\cT$ is $M$-acute for any $M\in S_\dim^+$ such that $\kappa(M) \leq \kappa(\cT)$.
\end{proposition}

\begin{proof}
Let $u,v$ be two non-zero vertices of a common simplex $T\in \cT$, and let $M\in S_\dim^+$. Let $u' := u/\|u\|$ and let $v' := v/\|v\|$. By construction one has
$$
\|u'+v'\|^2 = 2(1+\<u', v'\>) \geq 2(1+\gamma(\cT)), \qquad \|u'-v'\|^2 =  2(1-\<u',v'\>) \leq 2(1-\gamma(\cT)).
$$
Let us assume for contradiction that $\<u, M v\> < 0$, which implies that $\|u'+v'\|_M < \|u'-v'\|_M$. Observing that 
$$
\kappa(M)^2 = \|M\| \|M^{-1}\| \geq \frac {\|u'-v'\|^2_M}{\|u'-v'\|^2} \frac{\|u'+v'\|^2}{\|u'+v'\|^2_M} >\frac{1+\gamma(\cT)}{1-\gamma(\cT)},
$$
we obtain that $\kappa(M) > \kappa(\cT)$, which concludes the proof of this proposition.
\end{proof}

Our construction of $M$-acute meshes relies on special coordinate systems in the grid $\Z^m$, adapted to the anisotropic geometry encoded by $M\in S_m^+$. 
\begin{definition}[Bases and superbases]
\label{def:BasisSuperbase}
A basis of $\Z^m$ is an $m$-plet $(e_1, \cdots, e_m) \in (\Z^m)^m$ such that $|\det(e_1, \cdots, e_m)| = 1$. A superbase of $\Z^m$ is an $m+1$-plet $(e_0, \cdots, e_m)$ such that $e_0+\cdots+e_m = 0$, and $(e_1, \cdots, e_m)$ is a basis of $\Z^m$. 
\end{definition}
\begin{definition}
\label{def:ObtuseSuperbase}
A superbase $(e_0, \cdots, e_m)$ of $\Z^m$ is said to be $M$-obtuse iff $\<e_i, M e_j\> \leq 0$ for all $0 \leq i < j \leq m$.
\end{definition}
There exists for each $M \in S_m^+$, $m \in \{2,3\}$ at least one $M$-obtuse superbase \cite{CS92}. 
The construction of 2D and 3D $M$-obtuse superbases relies on lattice basis reduction algorithms  \cite{L1773,NS09} (hence the name of our numerical scheme) and has cost $\cO(\ln \kappa(M))$, see \S \ref{sec:reduced}. 
We arrive at the main contribution of this paper: the FM-LBR  stencils, which are causal and of bounded cardinality. 

\begin{figure}
\begin{center}
\iftoggle{siam}{
\includegraphics[width=2.5cm]{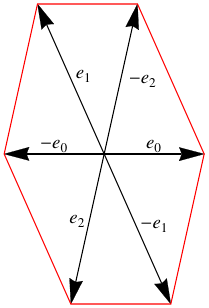}
\includegraphics[width=3.5cm]{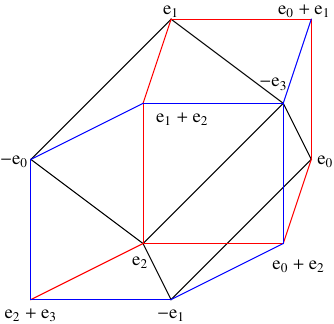}
\hspace{0.2cm}
\includegraphics[width=3cm]{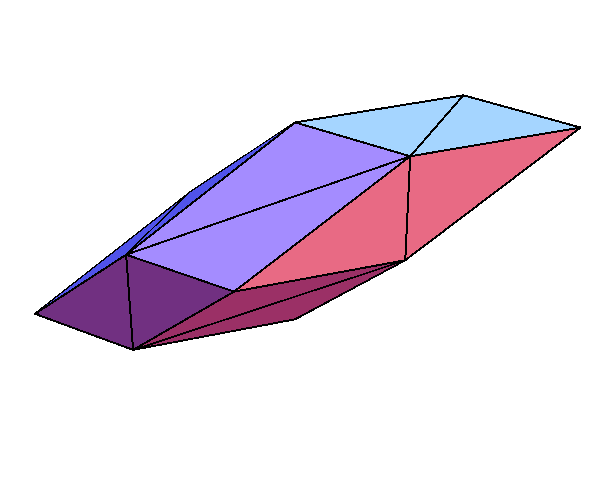}
\hspace{-0.3cm}
\includegraphics[width=3cm]{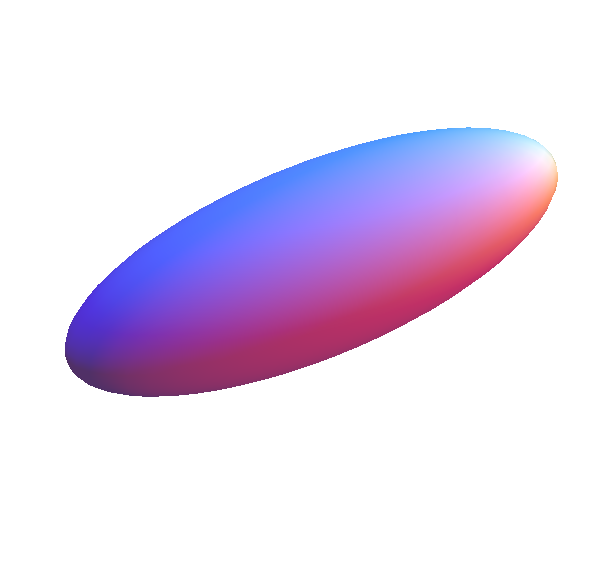}
}{
\includegraphics[width=3cm]{\pathPic FastMarchingIllus/RevisedStencils/Stencil2D_AllVertices.pdf}
\includegraphics[width=4.5cm]{\pathPic FastMarchingIllus/RevisedStencils/Stencil3D_AllVertices.pdf}
\hspace{0.2cm}
\includegraphics[width=4cm]{\pathPic FastMarchingIllus/RevisedStencils/Stencil3D_kappa=3.png}
\hspace{-0.3cm}
\includegraphics[width=4cm]{\pathPic FastMarchingIllus/RevisedStencils/Ellipse3D_kappa=3.png}
}
\end{center}
\caption{\label{figNeigh}
Connectivity of the 2D (left) and 3D (center left, interior edges omitted, missing boundary edges are symmetric w.r.t.\ the origin) FM-LBR meshes $\cT(M)$, see Proposition \ref{prop:Stencils} (note that $- e_i= \sum_{j \neq i} e_j$).
Mesh $\cT(M)$ (center right) associated to $M \in S_3^+$ of eigenvalues $3^2,1,1$, and eigenvector $(3,1,2)$ for the first eigenvalue. Right: unit ball for the norm $\|\cdot\|_M$. 
}
\end{figure}

\begin{proposition}[The FM-LBR acute meshes]
\label{prop:Stencils}
Let $M \in S_m^+$, and let $(e_0,\cdots,e_m)$ be an $M$-obtuse superbase of $\Z^m$ (if one exists). An $M$-acute mesh $\cT(M)$ is obtained by collecting the simplices of vertices $(\sum_{i=0}^{k-1} b_i)_{k=0}^{m}$ associated to all $(m+1)!$ permutations $(b_i)_{i=0}^m$ of $(e_i)_{i=0}^m$. It has $2^{m+1} -1$ vertices.
\end{proposition}

\begin{proof}
Proof of the properties of the simplices. 
Let $(b_i)_{i=0}^m$ be a permutation of $(e_i)_{i=0}^m$, and let $v_k := \sum_{i=0}^{k-1} b_i$, for all $0 \leq k \leq m$. Clearly (i) $v_0=0$, and (ii)  $|\det(v_1, \cdots, v_m)| = |\det(b_1, \cdots, b_m)| =  |\det(e_1, \cdots, e_m)|= 1$.
Acuteness condition (iii): for any $1 \leq k < l \leq m$ one has $v_l = -\sum_{j=l}^m b_j$ by Definition \ref{def:BasisSuperbase}, hence by Definition \ref{def:ObtuseSuperbase}
\begin{equation*}
\<v_k, M v_l\> = - 
\sum_{0 \leq i < k} \ \sum_{l \leq j \leq m} \<e_i,Me_j\> \geq 0.
\end{equation*}
Proof that $\cT(M)$ is a conforming mesh covering a neighborhood of the origin. Consider the $m+1$-dimensional Kuhn simplices 
$T_\sigma := \{ (\lambda_0, \cdots, \lambda_m); \, 0 \leq \lambda_{\sigma(0)} \leq \lambda_{\sigma(1)} \leq \cdots \leq \lambda_{\sigma(m)} \leq 1\}$, associated to all permutations $\sigma$ of $\{0,\cdots,m\}$, which form a conforming mesh $\cT_0$ of $[0,1]^{m+1}$. The linear map $A : \R^{m+1} \to \R^m$, defined by 
$A(\lambda_0, \cdots, \lambda_m) := -\sum_{i=0}^m \lambda_i e_i$, has a kernel generated by $v_1 := (1, \cdots,1)$. It sends $[0,1]^{m+1}$ onto a neighborhood of $0 \in \R^m$, and transforms $\cT_0$ into an $m$-dimensional conforming mesh $\cT_M$ of this neighborhood, collapsing onto $0$ the common edge $[0,v_1]$ of the simplices $T_\sigma$. Other vertices of $[0,1]^{m+1}$ have pairwise distinct images by $A$, since their difference is not proportional to $v_1$; hence $\cT_M$ has $2^{m+1}-1$ vertices (one less than $\cT_0$). 
Noticing the identity $-\sum_{i=0}^m \lambda_i e_i = \sum_{k=1}^m (\lambda_{\sigma(k)} - \lambda_{\sigma(k-1)}) \sum_{i=0}^{k-1} e_{\sigma(k)}$, we find that the image of $T_\sigma$ by $A$ is the simplex of $\cT(M)$ associated to the permuted superbase $(e_{\sigma(i)})_{i=0}^m$, hence $\cT_M=\cT(M)$, which concludes the proof.
\end{proof}

Obtuse superbases are similarly used in \cite{FM13} to produce sparse non-negative stencils for Anisotropic Diffusion (AD-LBR scheme), with different results: 3D stencils have $12$ non-zero vertices in \cite{FM13}, and $14$ here.
This illustrates the versatility of this concept, which can be used to design stencils satisfying various geometric properties: an acuteness condition for the FM-LBR (implying the scheme causality), and the non-negative decomposition of a tensor for the AD-LBR (guaranteeing the scheme monotonicity).

\begin{definition}
\label{def:Admissibility}
A family of meshes $(\cT(x))_{x \in \Omega}$, defined for all points of an open domain $\Omega \subset \R^m$ is admissible iff there exists two constants $0 < r$ and $R < \infty$ such that: for each $x\in \Omega$
\begin{itemize}
\item The mesh $\cT(x)$ covers neighborhood of $0$, and its vertices belong to $\Z^m$.
\item (Boundedness) The vertices $e$ of $\cT(x)$ satisfy $\|e\| \leq R$.
\item (Stability) There exists a basis $\cB(x)$ of $\Z^m$ which elements, and their opposites, are vertices of $\cT(y)$ for all $y \in \Omega$ such that $\|x-y\| < r$.
\end{itemize}
\end{definition}

\begin{figure}
\includegraphics[width=3cm]{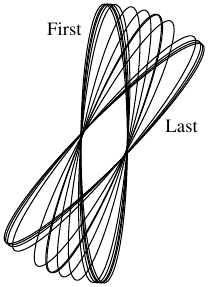}
{\raise \iftoggle{siam}{0.5cm}{0cm} \hbox{
\includegraphics[width= \iftoggle{siam}{10cm}{13cm}]{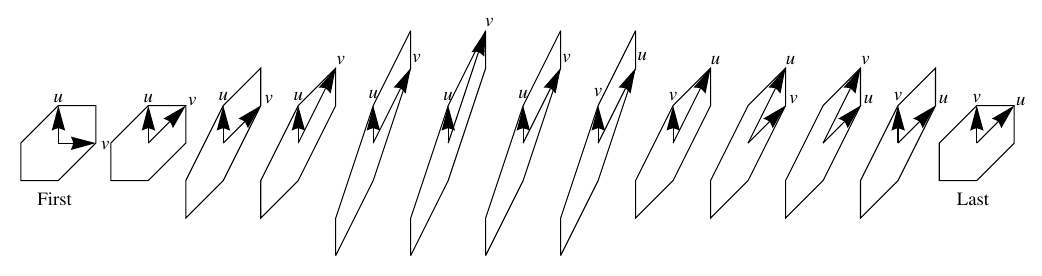}
}}
\caption{
\label{fig:BasisRotate}
The unit sphere $\{u\in \R^2; \, \|u\|_M=1\}$, an $M$-reduced basis $(u,v)$, and the boundary of the FM-LBR mesh $\cT(M)$, for some $M\in S_2^+$ of anisotropy ratio $\kappa(M)=6$, and eigenvector $(\cos\theta,\, \sin\theta)$, $\theta\in [\pi/4,\pi/2]$, associated to the small eigenvalue.
}
\end{figure}

A fortunate but non-trivial fact is that the FM-LBR family of meshes $(\cT(\cM(x)))_{x\in \Omega}$ is admissible, see Proposition \ref{prop:Admissibility} and Figures \ref{fig:BasisRotate} and \ref{fig:BasisScale}. This property implies the FM-LBR consistence, using Proposition \ref{prop:Convergence} which is a minor extension of the convergence results in \cite{T95,BR06}.

\begin{proposition}
\label{prop:Admissibility}
Let $\Omega \subset \R^m$ be open and bounded, and let $\cM\in C^0(\overline \Omega, S_m^+)$.
If $m \leq 3$ then the FM-LBR family of meshes $(\cT(\cM(x)))_{x \in \Omega}$ is admissible (with Boundedness constant $R=C_m \kappa(\cM)$, $C_2=2$, $C_3=4$). More generally, if $m \leq 4$ then any family of meshes $(\cT(x))_{x \in \Omega}$ such that $\cT(x)$ is $\cM(x)$-acute for all $x\in \Omega$, satisfies the (Stability) property.
\end{proposition}

\begin{proposition}
\label{prop:Convergence}
Let $\Omega \subset \R^m$  be an open bounded set, equipped with a Riemannian metric $\cM \in C^0(\overline \Omega, S_m^+)$, and an admissible family $(\cT(z))_{z \in \Omega}$ of meshes. For all $h>0$ let $Z_h := h \Z^2$, and for all $z \in Z_h \cap \Omega$ consider the stencil $V_h(z) := z+ h \cT(z)$. Then the solutions $\dist_h : Z_h \to \R \cup \{+\infty\}$ of the discrete system \eqref{discreteSys} converge uniformly as $h \to 0$ to the viscosity solution $\distC$ of the eikonal equation \eqref{eikonal}:
\begin{equation}
\label{eq:UnifConv}
\lim_{h \to 0} \, \max_{z \in Z_h \cap \Omega} |\dist_h(z) - \distC(z)| = 0.
\end{equation}
\end{proposition}

To sum up the FM-LBR strengths, this algorithm is universally consistent, has a competitive accuracy w.r.t.\ alternative methods, and its computational cost is almost unaffected by the Riemannian metric anisotropy.
For fairness we discuss below the potential downsides of our original stencil construction, which limits the potential for generalization (efficiency is as often at the cost of specialization), and impacts memory usage.
\begin{itemize}
\item (Finsler metrics)
The FM-LBR only applies to the anisotropic eikonal equation associated to a Riemannian metric, while other methods such as the AGSI, OUM, MAOUM \cite{BR06,SV03,AltonMitchell12} can handle more general Finsler metrics. Indeed the structure $\|\cdot\|_{\cM(z)}$ of the local norm at $z$ associated to a Riemannian metric is required in the FM-LBR stencil construction Proposition \ref{prop:Stencils}, which involves an $\cM(z)$-obtuse superbase. Finsler metrics are in contrast defined by arbitrary asymmetric norms  $|\cdot|_z$, depending continuously on $z$.


Constructing causal static stencils is an active subject of research in the case of Finsler metrics. A characterization obtained in \cite{VladMSSP}, was used in \cite{M12c} to develop the FM-ASR (Fast Marching using Anisotropic Stencil Refinement), which is close in spirit and in efficiency to the FM-LBR, but has a different application range, since it handles Finsler metrics on two dimensional domains. 
It was also observed in \cite{SethBook96,SV03,AltonMitchell08} that the canonical stencil (Figure \ref{fig:Classical}, left) is causal for Finsler metrics featuring only axis-aligned anisotropy, in such way that the original fast marching algorithm can be applied. 

\item (Domain discretization)
The FM-LBR requires a cartesian grid discretization. In contrast an important research effort \cite{BS98,BR06,KS98,LYC03,SV03} has been devoted to the more difficult setting of meshed domains, with an unstructured set of vertices. 
These methods are natural candidates for the computation of distances on a manifold $\cS \subset \R^m$, while the FM-LBR would require a (collection of) local chart(s) $\vp : \Omega \to \cS$ equipped with the metric $\cM(z) := \nabla \vp(z)^\trans \nabla \vp(z)$.
The FM-LBR heavily relies on the cartesian grid arithmetic structure, through the concept of obtuse superbases.

\item (Boundary conditions)
The null boundary conditions chosen in the eikonal equation \eqref{eikonal} can be replaced with Dirichlet data $\distC_0 : \partial \Omega \to \R$, of $1$-Lipschitz regularity \cite{L82} with respect to the Riemannian distance $\distC(\cdot, \cdot)$, see \eqref{def:Length}. 
In that event, Algorithm \ref{algo:FastMarching} requires to extend the boundary data $\distC_0$ to a ghost layer, containing all grid points $y \in Z \sm \Omega$ such that the reverse stencil $V[y]$ is non-empty. The FM-LBR uses large stencils, of euclidean radius $\cO(h \kappa(\cM))$ on a grid \eqref{def:Grid} of scale $h$, which complicates this extension in contrast with e.g.\ the AGSI \cite{BR06} of stencil radius $\cO(h)$. 

Outflow boundary conditions, on a portion $\Gamma \subsetneq \partial \Omega$ of the domain's boundary, are natural in applications, see \S \ref{sec:num}. They are implemented by excluding in the definition \eqref{def:HopfLax} of $\Lambda(\dist, x)$, faces of $\partial V(x)$ containing an exterior vertex $y \in Z \sm \Omega$ close to $\Gamma$. If $x$ lies in a corner of $\Omega$, and if the stencil $V(x)$ is strongly anisotropic, then it may happen that all the vertices of $\partial V(x)$ lie \emph{outside} $\Omega$, so that the solution of \eqref{discreteSys} satisfies $\dist(x) = +\infty$. 
(In our experiments \S \ref{sec:num}, this happened in the square domain's corners when test case 1 was rotated by an angle $\theta \in [0.56,0.61]$ radians. These four infinite values were rejected when estimating numerical errors.) 
Despite these minor inconveniences, the FM-LBR behaves remarkably well in our numerical experiments, \S \ref{sec:num} and Appendix \ref{subsec:Accuracy}, with these more general boundary conditions.

\item (Dimension)
The FM-LBR causal stencils construction, see Proposition \ref{prop:Stencils}, is limited to domains of dimension $2$ and $3$, because some matrices $M \in S_4^+$ do not admit any $M$-obtuse superbase \cite{NS09}. 
An alternative construction of $M$-acute meshes of uniformly bounded cardinality is proposed in \cite{M12b}, for all $M \in S_4^+$; this cardinality is not small unfortunately, with $768$ simplices. 

The FM-LBR however extends in a straightforward manner to Riemannian metrics $\cM$ having a block diagonal structure, with blocks of size $1$, $2$ or $3$, using the following construction.
For $i \in \{1,2\}$ let $m_i$ be a positive integer, let $M_i \in S_{m_i}^+$, and let $\cT_i$ be an $M_i$-acute mesh. Let $m := m_1+m_2$ and let $M\in S_m^+$ be the matrix of diagonal blocks $M_1,M_2$. An $M$-acute mesh $\cT$ is obtained by collecting the $m$-dimensional simplices of vertices $(0,0)$, $(u_1,0), \cdots, (u_{m_1},0)$, $(0,v_1), \cdots, (0,v_{m_2})$, where the simplices of vertices $0, u_1, \cdots, u_{m_1}$ and $0,v_1, \cdots, v_{m_2}$ belong to $\cT_1$ and $\cT_2$ respectively.

Block diagonal metrics are not uncommon in the context of medical imaging \cite{BC10}, as they inherit the cartesian product structure of the fast marching domain: $\Omega = \Omega_0 \times \Omega_1$, where $\Omega_0$ is a physical domain of dimension $\leq 3$, and $\Omega_1$ is an abstract parameter domain of dimension $\leq 2$.
\end{itemize}

\begin{figure}
\begin{center}
\begin{tabular}{cc}
{\raise \iftoggle{siam}{0.2cm}{0.5cm} \hbox{
\includegraphics[width=2cm]{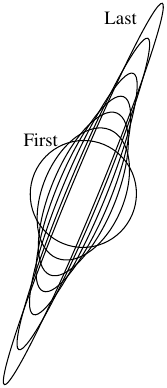}
}}
&
\includegraphics[width= \iftoggle{siam}{7cm}{8cm}]{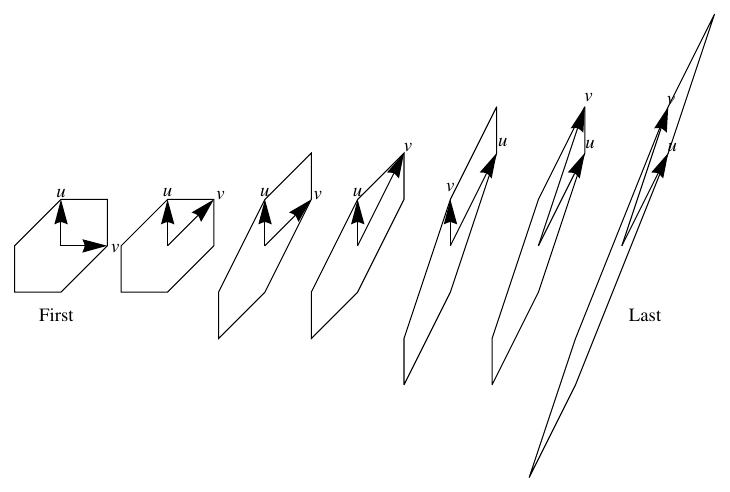}
\end{tabular}
\end{center}
\iftoggle{siam}{}{\vspace{-0.8cm}}
\caption{
\label{fig:BasisScale}
The unit sphere $\{u\in \R^2; \, \|u\|_M=1\}$, an $M$-reduced basis $(u,v)$, and the boundary of the FM-LBR mesh $\cT(M)$, for some $M\in S_2^+$ of anisotropy ratio $\kappa(M)$ ranging from $1$ to $15$, and eigenvector $(\cos(3\pi/8),\sin(3 \pi/8))$ associated to the small eigenvalue.
}
\end{figure}

\paragraph{Outline.}
Further insight on the FM-LBR is given in \S \ref{sec:reduced}, including the proof of Propositions \ref{prop:Admissibility} and \ref{prop:Convergence}. 
Numerical experiments are presented in \S \ref{sec:num}. In addition, Appendix \ref{sec:Paths} describes a robust minimal path extraction method for the FM-LBR and other Dijkstra inspired solvers of the eikonal equation. A heuristic analysis of the FM-LBR accuracy, and a last numerical experiment, appear in Appendix \ref{subsec:Accuracy}.

\begin{remark}[Memory requirements]
\label{rem:memory}
The memory requirements of numerical methods for the eikonal equation, such as the AGSI, the OUM and the FM-LBR, are dominated by (I) storing  the discrete solution $\dist$ and the Riemannian metric $\cM$, sampled on the discrete domain $\Omega \cap Z$, and (II) storing the graph structure underlying the numerical scheme. Point (I) requires two tables of $N$ and  $N \dim (\dim+1)/2$ reals, which may be represented in 32 bit (single precision) or 64 bit (double precision) format. The storage cost for the metric can be avoided if it has an analytical expression.

Point (II) can be avoided for the AGSI and the OUM when these methods are executed on a mesh with a trivial periodic structure, which is the case in our experiments. 
For the FM-LBR, Point (II) amounts to storing the non-empty reverse stencils $V[y]$, at all points of $Y := \{y \in Z; \, V[y]\neq \emptyset\}$, since the direct ones can be recomputed individually on demand for a minor cost. The set $Y$ the union of $\Omega \cap Z$ and of a thin boundary layer (in our experiments, $Y=\Omega \cap Z$ due to the use of outflow boundary conditions). 
The chosen data structure uses two tables: one of vectors (the differences $x-y$, for $x \in V[y]$, $y \in Y$, enumerated consecutively), and one of $\#(Y) \approx N$ integers (the start, for each $y \in Y$, of the description of $V[y]$ in the previous table). We represent integers in 32bit format, and vector components in 8bit format, since these are small integers by construction. 


Summing up, we find that the memory requirements of the FM-LBR are larger than those of the AGSI or the OUM (on a grid), by a factor ranging from $2$ (metric and solution stored in double precision), to $9$ (analytical metric, solution stored in single precision), through $3$ (metric and solution stored in single precision), in two dimensions. Respectively, in three dimensions, from $2.6$ to $24$, through $4.3$. 
\end{remark}

\section{Analysis of the FM-LBR}
\label{sec:reduced}

We introduce in \S \ref{sec:basis} the concepts of Lattice Basis Reduction. They are used in \S \ref{sec:MeshProperties} to estimate the construction cost of the FM-LBR meshes $\cT(M)$, and to prove their admissibility in the sense of Definition \ref{def:Admissibility}, as announced in Proposition \ref{prop:Admissibility}. We finally prove in \S \ref{sec:Convergence} the announced convergence result Proposition \ref{prop:Convergence}.

\subsection{Introduction to Lattice Basis Reduction}
\label{sec:basis}

We briefly introduce the framework of (low dimensional) Lattice Basis Reduction, used in the next subsection to construct and demonstrate the properties of $M$-acute meshes. 
See \cite{NS09} and references therein for more details on this rich theory, from which we use only one result: Theorem \ref{th:Red} stated below.
We denote by $b_1 \Z +\cdots + b_k \Z$ the sub-lattice of $\Z^\dim$ generated by $b_1, \cdots, b_k \in \Z^\dim$:  
$$
b_1 \Z +\cdots + b_k \Z := \{b_1 z_1 + \cdots + b_k z_k ; \ z_1, \cdots, z_k \in \Z\}.
$$
If $k=0$ then the above sum equals $\{0\}$ by convention. 
\begin{definition}
\label{def:RedBasis}
Let $1 \leq m \leq 4$ and let $M \in S_m^+$. 
A basis $(b_1, \cdots , b_\dim)$ of $\Z^\dim$ is said $M$-reduced iff for all $1 \leq k \leq \dim$:
\be
\label{def:uk}
b_k \in \argmin\{ \|z\|_M; \, z\in \Z^\dim \sm (b_1 \Z+ \cdots +b_{k-1} \Z)\}.
\ee
\end{definition}

If $M\in S_\dim^+$ is a diagonal matrix of coefficients $(\lambda_1, \cdots, \lambda_\dim)$, and if $0<\lambda_{\sigma(1)} \leq \cdots \leq \lambda_{\sigma(\dim)}$ for some permutation $\sigma$, then the permutation $(e_{\sigma(i)})_{i=1}^m$ of the canonical basis of $\Z^m$ is an $M$-reduced basis of $\Z^\dim$. 
See Figures \ref{figNeigh}, \ref{fig:BasisRotate}, \ref{fig:BasisScale}, for some examples of $M$-reduced bases associated to non-diagonal matrices $M$.
In dimension $\dim\geq 5$ there exists matrices $M\in S_\dim^+$ such that no basis of $\Z^\dim$ satisfies the relations \eqref{def:uk}, see \cite{NS09} (these relations state that $\|b_i\|_M$ equals the $i$-th Minkowski's minimum $\lambda_i(M)$). Minkowski's reduction \cite{NS09} is the natural generalization of Definition \ref{def:RedBasis} in dimension $\dim \geq 5$.

\begin{theorem}[Nguyen, Stelh\'e, 2009]
\label{th:Red}
There exists an algorithm which, given a matrix $M\in S_\dim^+$ as input, $1 \leq \dim \leq 4$, produces an $M$-reduced basis of $\Z^\dim$ and has the numerical cost $\cO(1+\ln \kappa(M))$. 
\end{theorem}

\begin{proof}
The proof is contained in \cite{NS09}, and we only point out here the precise reference within the paper and the slight differences in notations. 
The algorithm described in \cite{NS09} takes as input a basis $(b_1, \cdots, b_\dim)$ (here: the canonical basis of $\R^\dim$) of a lattice $L$ (here: $\Z^\dim$), and its Gram matrix with respect to some scalar product (here: the Gram matrix is $M$).  The algorithm outputs a greedy reduced basis of the lattice $L$, a notion which coincides with Minkowski's reduction if $\dim\leq 4$ (Lemma 4.3.2 in \cite{NS09}), which itself coincides with Definition \ref{def:RedBasis} if $\dim\leq 4$.

The main loop of the iterative algorithm is executed at most the following number of times (Theorem 6.0.5 in \cite{NS09}): 
$$
\cO\left( 1 + \ln \max_{1 \leq i \leq \dim} \|b_i\|_M - \ln \min_{u\in L} \|u\|_M \right),
$$
hence $\cO(1+ \ln \|M\|^\frac 1 2-\ln \|M^{-1}\|^{-\frac 1 2}) = \cO(1+\ln \kappa(M))$ times in our setting.
The complexity of each of these iterations is dominated by a closest vector search, described in Theorem 5.0.4 in \cite{NS09}, which consists of the inversion of a $k\times k$ Gram matrix, where $1 \leq k \leq \dim-1$, and an exhaustive search among $\cO(1)$ candidate vectors. In terms of elementary operations ($+,-,\times, /$) among reals, each iteration of this algorithm thus has cost $\cO(1)$, and the overall cost is the number of iterations $\cO(1+\ln \kappa(M))$.

Note that an important part of the discussion in \cite{NS09} is devoted to the special case where the vectors $(b_1, \cdots, b_\dim)$ have \emph{large integer} coefficients, the Gram matrix is computed with respect to the standard euclidean scalar product, and the complexity of an elementary operation ($+,-,\times, /$) among integers is not $\cO(1)$ but depends on the size of these integers. This more subtle notion of complexity, named bit complexity, is not relevant in our setting.
\end{proof}

The two dimensional version of the algorithms mentioned in Theorem \ref{th:Red} dates back to Lagrange \cite{L1773}, and mimicks the search for the greatest common divisor of two integers. 
This algorithm uses only a pair $(u,v)$ of (mutable) variables in $\Z^2$, initialized as the canonical basis of $\R^2$. The pair $(u,v)$ becomes an $M$-reduced basis at the end of the following loop, which takes at most $\cO(\ln \kappa(M))$ iterations. $\mathrm{Round}$ denotes rounding to a closest integer.
\begin{center}
\begin{tabular}{l}
\bf Do $(u,v) \, \assign \, (v, \ u - \mathrm{Round}(\<u, M v\>/\|v\|_M^2 ) \,v)$,\\
\bf While $\|u\|_M > \|v\|_M$.
\end{tabular}
\end{center}

We end this subsection with a basic estimate of the norms and scalar products of the elements of $M$-reduced bases.
\begin{proposition}
\label{prop:normRed}
Let $1 \leq \dim \leq 4$, let $M\in S_\dim^+$ and let $(b_1, \cdots, b_\dim)$ be an $M$-reduced basis of $\Z^\dim$.
Then for any $1 \leq i \leq \dim$
\begin{equation}
\label{biNorm}
\|b_i\| \leq \kappa(M), \quad \text{ and } \quad \|b_i\|_M \leq \kappa(M)\|b_1\|_M.
\end{equation}
For any integer combination $z \in b_1 \Z + \cdots + b_{i-1} \Z + b_{i+1} \Z + \cdots + b_m \Z$, of the basis elements distinct from $b_i$, one has 
\be
\label{scalNormIneq}
2 |\<b_i, M z\>| \leq \|z\|_M^2.
\ee
\end{proposition}

\begin{proof}
Proof of \ref{biNorm}.
Let $(e_j)_{j=1}^m$ denote the canonical basis of $\R^m$. By a dimensionality argument there exists $1 \leq j \leq \dim$ such that $e_j \notin b_1\Z+\cdots + \cdots b_{i-1} \Z$. By construction \eqref{def:uk} we obtain $\|b_i\|_M \leq \|e_j\|_M \leq \|M\|^\frac 1 2$ since $\|e_j\|=1$, hence $\|b_i\| \leq \kappa(M)$ as announced in \eqref{biNorm}. Observing that  $\|b_1\|_M \geq \|M^{-1}\|^{-\frac 1 2}$ since $\|b_1\| \geq 1$, and recalling that $\|b_i\|_M \leq \|M\|^\frac 1 2$, we obtain the second announced estimate. 

Proof of \eqref{scalNormIneq}. Remark that $b_i+z\notin b_1 \Z+ \cdots+b_{i-1} \Z$, since otherwise the basis element $b_i$ would be a linear combination of $b_1, \cdots, b_{i-1}, b_{i+1}, \cdots, b_\dim$.
Definition \ref{def:RedBasis} thus implies:
$$
\|b_i\|^2_M \leq \|b_i+z\|^2_M = \|b_i\|_M^2+2 \<b_i, M z\>+\|z\|_M^2,
$$
hence $-2 \<b_i, M z\> \leq \|z\|_M^2$.
Likewise $2 \<b_i, M z\> \leq \|z\|_M^2$, which concludes the proof. 
\end{proof}


\subsection{Properties of $M$-acute meshes}
\label{sec:MeshProperties}
\paragraph{The FM-LBR stencil construction has numerical cost $\cO(1+ \ln \kappa(M))$.\\}
The construction of an FM-LBR mesh $\cT(M)$, $M \in S_m^+$, has unit cost (for fixed $m$) given an $M$-obtuse superbase, see Proposition \ref{prop:Stencils}. We give below a unit cost construction of an $M$-obtuse superbase given an $M$-reduced basis, in dimension $m \in \{2,3\}$, which by Theorem \ref{th:Red} can itself be obtained at cost $\cO(1+\ln \kappa(M))$.

Dimension 2. Let $(b_1,b_2)$ be an $M$-reduced basis. Up to replacing $b_2$ with its opposite, we may assume that  $\<b_1, M b_2 \> \leq 0$. Using \eqref{scalNormIneq} we obtain $2|\<b_1, M b_2\>| \leq \|b_1\|_M^2$, hence $\<M b_1,-b_1-b_2\> \leq 0$ and likewise $\<M b_2,-b_1-b_2\> \leq 0$, so that $(b_1,b_2, -b_1-b_2)$ is an $M$-obtuse superbase. 

Dimension $3$: we reproduce without proof the construction of \cite{FM13}. Let $(b_1,b_2,b_3)$ be the elements of an $M$-reduced basis, permuted and signed so that $|\<b_1, M b_2\>| \leq \min \{ - \<b_1, M b_3\>, \allowbreak -\<b_2, M b_3\>\}$. An $M$-obtuse basis is given by $(b_1,b_2,b_3,-b_1-b_2-b_3)$ if $\<b_1, M b_2\> \leq 0$, and $(-b_1,b_2,b_1+b_3,-b_2-b_3)$ otherwise.


\paragraph{Radius of the FM-LBR stencils.\\}

Assume that an FM-LBR acute mesh $\cT(M)$ is built from an $M$-obtuse superbase obtained as in the previous paragraph\footnote{This implementation is used in our numerical experiments. See Corollary \ref{corol:VertexNorm} for a proof with no assumption.} 
from an $M$-reduced basis $(b_i)_{i=1}^m$. Then one easily checks that its vertices have the form $e = \sum_{i=1}^m \ve_i b_i$, where $\ve_i \in \{-1,0,1\}$ and $(b_i)_{i=1}^m$ is the used $M$-reduced basis. Hence the vertices norm  $\|e\| \leq \sum_{i=1}^m \|b_i\| \leq m \kappa(M)$ obeys the announced bound.

An $M$-reduced basis of $\Z^\dim$ contains by definition \emph{small} vectors with respect to the norm $\|\cdot\|_M$: the smallest linearly independent ones with integer coordinates. As a result, the FM-LBR meshes $\cT(M)$ have a small radius with respect to the norm $\|\cdot \|_M$. In constrast, these meshes can be large from the euclidean perspective.
This has consequences on the FM-LBR accuracy, see Appendix \ref{subsec:Accuracy}.

\paragraph{Stability of $M$-acute meshes.\\}

Consider the following distance on the set $S_\dim^+$ of symmetric positive definite matrices: for all $M,N\in S_\dim^+$ 
\begin{equation}
\label{def:dtimes}
\dtimes(M,N) := \sup_{u\neq 0} \left|\ln \|u\|_M - \ln \|u\|_N\right|.
\end{equation}
This distance allows to compare the norms of vectors multiplicatively, in contrast with the classical operator norm which is tailored for additive comparisons. Indeed denoting $\alpha := \dtimes(M,N)$ and $\beta := \|M^\frac 1 2 - N^\frac 1 2\|$, one has for all $u\in \R^2$ such that $\|u\| = 1$
$$
e^{-\alpha} \leq \|u\|_M/\|u\|_N \leq e^{\alpha}, \stext{ and } -\beta \leq \|u\|_M - \|u\|_N \leq \beta.
$$

The next lemma establishes a lower bound on the $\|\cdot\|_M$ norm of points with integer coordinates outside of an $N$-reduced mesh, when the matrices $M,N\in S_\dim^+$ are close enough. 

\begin{lemma}
\label{lem:NormOutT}
Let $M,N\in S_\dim^+$, with $m\leq 4$. Let $u_1, \cdots, u_\dim$ be an arbitrary $M$-reduced basis of $\Z^\dim$, and let $\cT$ be an $N$-reduced mesh.  
Consider a point $z\in \Z^\dim$ which is not a vertex of $\cT$.
Then there exists $1 \leq l \leq \dim$ such that
$$
z \in u_1 \Z+ \cdots + u_l\Z \ \text{ and } \ \|z\|_M^2 \,e^{4 \dtimes(M,N)} \geq \|u_l\|_M^2 + \|u_1\|_M^2.
$$
\end{lemma}

\begin{proof}
Since $\cT$ covers a neighborhood of the origin, there exists a simplex $T\in \cT$ and a real $\lambda >0$ such that $\lambda z \in T$.
Denoting by $v_1, \cdots, v_\dim$ the non-zero vertices of $T$, there exists therefore non-negative reals $\alpha_1, \cdots, \alpha_\dim\in \R_+$ such that $z = \alpha_1 v_1 + \cdots + \alpha_\dim v_\dim$. Since $(v_1, \cdots, v_\dim)$ form a basis of $\Z^m$, these coefficients are integers.

Up to reordering the vertices $v_1, \cdots, v_\dim$, we may assume that $\alpha_1, \cdots, \alpha_k$ are positive, and that $\alpha_{k+1},\cdots, \alpha_\dim$ are zero, for some $1 \leq k \leq \dim$.  
Let $l$ be the smallest integer such that $z\in u_1 \Z+ \cdots +u_l \Z$. 
There exists $1 \leq i \leq k$ such that $v_i \notin u_1 \Z+ \cdots +u_{l-1} \Z$, hence $\|v_i\|_M\geq \|u_l\|_M$; other vertices satisfy $\|v_j\|_M \geq \|u_1\|_M$, $1\leq j \leq k$, since they are non-zero and have integer coordinates.
Therefore, denoting $\delta := \dtimes(M,N)$
\begin{align*}
e^{4 \delta} \|z\|_M^2 &\geq e^{2 \delta} \|z\|_N^2
 = e^{2 \delta}\left(\sum_{1 \leq i \leq k} \alpha_i^2 \|v_i\|_N^2 + 2\sum_{1 \leq i<j\leq k} \alpha_i\alpha_j \<v_i, N v_j\>\right)\\
&\geq   e^{2 \delta}\sum_{1 \leq i \leq k} \alpha_i^2 \|v_i\|_N^2
\geq   \sum_{1 \leq i \leq k} \alpha_i^2 \|v_i\|_M^2
\geq \|u_l\|_M^2 + \left( \sum_{1 \leq i \leq k} \alpha_i^2 -1\right)\|u_1\|_M^2.
\end{align*}
Observing that $\alpha_1^2+ \cdots+\alpha_k^2 \geq 2$, since $z$ is not a vertex of $T$, we
conclude the proof.
\end{proof}

We prove in the next corollary that the vertices of an $M$-acute mesh contain the elements of an $N$-reduced basis, and their opposites, when the matrices $M,N$ are sufficiently close. This result, and the compactness of $\overline \Omega$, immediately implies the second point of Proposition \ref{prop:Admissibility}.

\begin{corollary}
\label{corol:BasisVertex}
Let $M,N\in S_\dim^+$, with $m \leq 4$, be such that 
\be
\label{dtimeskappa}
\dtimes(M,N) < \ln(1+\kappa(M)^{-2})/4.
\ee
Let $(b_1, \cdots, b_\dim)$ be an $M$-reduced basis of $\Z^\dim$, and let $\cT$ be an $N$-reduced mesh. Then $b_1, \cdots, b_\dim$ and $-b_1, \cdots, -b_\dim$
are vertices of $\cT$.
\end{corollary}

\begin{proof}
Let $1 \leq l \leq m$. We have  $b_l \notin b_1 \Z + \cdots + b_{l-1} \Z$, since $(b_i)_{i=1}^m$ is a basis, and Proposition \ref{prop:normRed} implies
$$
\|b_l\|^2_M e^{4 \dtimes(M,N)} < \|b_l\|_M^2+ \kappa(M)^{-2} \|b_l\|_M^2 \leq  \|b_l\|_M^2+\|b_1\|_M^2.
$$
Hence $b_l$ is a vertex of $\cT$, and likewise $-b_l$, by Lemma \ref{lem:NormOutT}.
\end{proof}

We finally estimate the radius of the FM-LBR meshes $\cT(M)$, see Proposition \ref{prop:Stencils}, in terms of the condition number of the matrix $M \in S_m^+$. This concludes the proof of Proposition \ref{prop:Admissibility}. 

\begin{corollary}
\label{corol:VertexNorm}
Let $M \in S_m^+$, with $m \in \{2,3\}$. Then any vertex $e$ of $\cT(M)$ satisfies $\|e\| \leq C_m \kappa(M)$, with $C_2:=2$ and $C_3:=4$.
\end{corollary}

\begin{proof}
Given $m$ linearly independent vertices $(b_i)_{i=1}^m$ of $\cT(M)$, one can express any other vertex under the form $e=\sum_{i=1}^m \alpha_i b_i$. A simple check by exhaustive enumeration shows that $\sum_{i=1}^m |\alpha_i| \leq C_m$ (these coefficients are independent of $M$), hence $\|e\| \leq C_m \max \{\|b_i\|; \, 1 \leq i \leq m\}$. 
Applying Corollary \ref{corol:BasisVertex} with $M = N$, we find that the vertices of $\cT(M)$ contain an $M$-reduced basis $(b_i)_{i=1}^m$, which by Proposition \ref{prop:normRed} satisfies $\|b_i\| \leq \kappa(M)$ for all $1 \leq i \leq m$. This concludes the proof.
\end{proof}


\subsection{Convergence of the FM-LBR}
\label{sec:Convergence}
We prove in this section the uniform convergence of the discrete system \eqref{discreteSys} solutions towards the anisotropic eikonal PDE \eqref{eikonal} solution, under the assumptions of Proposition \ref{prop:Convergence} and with its notations. 
Following the steps of \cite{BR06}, we begin with a discrete Lipschitz regularity estimate for the maps $\dist_h : Z_h \to \R$.

\begin{lemma}
\label{lem:Lipschitz}
There exists constants $h_0 > 0$ and $C_0 < \infty$ such that for all $0 < h \leq h_0$ and all $x, y \in Z_h$ one has 
\begin{equation}
\label{eq:LipschitzDiscrete}
|\dist_h(x) - \dist_h(y) | \leq C_0 \|x-y\|.
\end{equation}
\end{lemma}

\begin{proof}

We prove below that $|\dist_h(x) - \dist_h(y) | \leq C_1 h$ when $\|x-y\| = h$, in other words when $x$ and $y$ are neighbors on the grid $Z_h := h \Z^m$. This immediately implies $|\dist_h(x) - \dist_h(y)| \leq C_1 \|x-y\|_1$, where $\|(\lambda_1, \cdots, \lambda_m) \|_1 := \sum_{i=1}^m |\lambda_i|$, hence also \eqref{eq:LipschitzDiscrete} with $C_0 := C_1 \sqrt m$.

If $x\notin \Omega$ and $y \notin \Omega$, then $\dist_h(x) = \dist_h(y)=0$, and the result is proved. Up to exchanging $x$ and $y$, we may therefore assume that $x \in \Omega$. Let $B := \cB(x)$ be the basis corresponding to the property (Stability) of the admissible family of meshes $(\cT(x))_{x \in \Omega}$, see Definition \ref{def:Admissibility}. 
Each element $e$ of $B$ is a vertex of $\cT(x)$, hence satisfies $\|e\| \leq R$, see Definition \ref{def:Admissibility}  (Boundedness).
We abusively regard $B$ as an $m\times m$ matrix which columns are the basis elements, and observe that $\|B\| \leq R \sqrt m$ and $|\det(B)| = 1$. Therefore $\|B^{-1}\| \leq \|B\|^{m-1}/|\det B| \leq C_2 := (R \sqrt m)^{m-1}$. 

Since we assumed $\|x-y\|=h$, we have $y=x+ \ve h e_j$, for some $\ve \in \{-1,1\}$, $1 \leq j \leq m$, and where $(e_j)_{j=1}^m$ denotes the canonical basis of $\R^m$. Hence denoting by $(b_i)_{i=1}^m$ the elements of the basis $B$, and by $(\alpha_{ij})_{i,j=1}^m$ the coefficients of $B^{-1}$, we obtain $y = x + \ve h\sum_{i=1}^m \alpha_{ij} b_i$. Let $s := \sum_{i=1}^m |\alpha_{ij}|$, and let $(x_k)_{k=0}^s$ be a finite sequence of points of $Z_h$ such that $x_0:=x$, $x_s:=y$, and $x_{k+1}-x_k \in \{\pm h b_i; \, 1 \leq i \leq m\}$ for all $0 \leq k < s$.

By Cauchy-Schwartz's inequality one has $s/ \sqrt m \leq (\sum_{i=1}^m \alpha^2_{ij})^\frac 1 2 \leq \|B^{-1}\| \leq C_2$, hence $\|x_k - x\| \leq h R C_2  \sqrt m$ for any $0 \leq k \leq s$. We choose the upper grid scale bound $h_0$ so that this constant smaller than the radius $r$ involved in property (Stability) of Definition \ref{def:Admissibility}. If $x_k \in \Omega$, then by property (Stability) a grid point $x_{l}$, with $|k-l| \leq 1$ and  $0\leq l \leq s$, is a vertex of the stencil $V_h(x)$. Hence inserting $x_l$ in \eqref{def:HopfLax} we obtain $\dist_h(x_k) = \Lambda_h(\dist_h, x_k) \leq \|x_k - x_l\|_{\cM(x)} + \dist_h(x_l) \leq h R C_3+ \dist_h(x_l)$, with $C_3 := \max \{\|\cM(z)\|^\frac 1 2; \, z \in \overline \Omega\}$. If $x_k \notin \Omega$ then $\dist_h(x_k) = 0$, hence obviously $\dist_h(x_k) \leq \dist_h(x_l)$. Exchanging the roles of of $k$ and $l$ we obtain $|\dist_h(x_k) - \dist_h(x_l)| \leq h R C_3$, hence by the triangular inequality $|\dist_h(x) - \dist_h(y) | \leq h R C_3 s \leq h R C_3 C_2 \sqrt m$, which concludes the proof.
\end{proof}

The rest of the proof is only sketched, since it amounts to a minor adaptation of Theorem 11 in \cite{BR06}. The only difference lies in the presence, in \cite{BR06}, of a canonical interpolation operator on $\Omega$ (piecewise linear interpolation on a prescribed mesh).
Denote by $\distC_h$ the bilinear interpolant%
\footnote{%
Other interpolation schemes could be used, such as piecewise linear interpolation on a trivial periodic mesh, provided one can control the Lipschitz regularity constant and the support of the interpolated function.
} 
of $\dist_h$ on the grid $Z_h$, and observe that $|\distC_h(x) - \distC_h(y)| \leq K C_0 \|x-y\|$, for all $x, y \in \R^2$ and all $0 < h \leq h_0$, where $K$ is an absolute constant depending only on the interpolation scheme, and $C_0$, $h_0$, are the constants of Lemma \ref{lem:Lipschitz}. 
Note also that $\supp(\distC_h) \subset \{z+e; \, z \in \overline \Omega, \, \|e\| \leq h \sqrt m\}$, hence by Lipschitz regularity $\distC_h$ is bounded uniformly independently of $h$, and therefore by Arzel\`a-Ascoli's theorem the family $(\distC_h)_{0 < h \leq h_0}$ is pre-compact. Considering an arbitrary converging sub-sequence $\distC_{h(n)}$, with $h(n) \to 0$ as $n \to \infty$, one observes that the limit is supported on $\overline \Omega$, and applying the arguments of Theorem 11 in \cite{BR06} that it is a viscosity solution of the eikonal PDE \eqref{eikonal}. Uniqueness of such a solution \cite{L82} implies the pointwise convergence $\distC_h(x) \to \distC(x)$, as $h \to 0$, for all $x \in \overline \Omega$. Finally the announced uniform convergence \eqref{eq:UnifConv} follows from the uniform $KC_0$-Lipschitz regularity of $\distC_h$, for all $0 < h \leq h_0$.

\section{Numerical experiments}
\label{sec:num}

%


We compare numerically the FM-LBR with two popular solvers (AGSI, FM-8) of the eikonal equation, which enjoy a reputation of simplicity and efficiency in applications, and with the recent and closely related MAOUM.
The Adaptive Gauss Seidel Iteration\footnote{%
As suggested in \cite{BR06}, the stopping criterion tolerance for the AGSI iterations is set to $10^{-8}$.
}
(AGSI) \cite{BR06} produces numerical approximations which are guaranteed to converge towards the solution of the continuous anisotropic eikonal equation as one refines the computation grid\footnote{%
The grid is triangulated with a trivial periodic mesh, for the AGSI and the MAOUM. As a result the AGSI uses a 6 point stencil.%
%
}%
, for an arbitrary continuous Riemannian metric $\cM$. Fast Marching using the 8 point stencil (FM-8, stencil illustrated on Figure \ref{fig:Classical}, center left) does not offer this convergence guarantee, but has a quasi-linear complexity $\cO(N \ln N)$, in contrast with the super-linear complexity $\cO(\mu(\cM) N^{1+\frac 1 \dim})$ of the AGSI. Fast Marching using Lattice Basis Reduction%
\footnote{%
The FM-LBR stencil $V(z)$, at each grid point $z \in \Omega \cap Z$, is built \eqref{eq:VFromT} from the $\cM(z)$-reduced mesh $\cT(\cM(z))$ of Proposition \ref{prop:Stencils}, except if the matrix $\cM(z)$ is detected to be exactly diagonal. In that case we use the standard $4$ vertices neighborhood in 2D (resp. $6$ vertices in 3D), which is an $\cM(z)$-reduced mesh, see Figure \ref{fig:Classical} (left and center right). This modification has little impact on accuracy or CPU time, but avoids to pointlessly break symmetry. %
}
(FM-LBR) aims to offer the best of both worlds: a convergence guarantee, and fast computation times%
\footnote{%
Note that the FM-LBR memory requirement is higher than that of the AGSI and FM-8, see Remark \ref{rem:memory}.
}.

We also implemented the Monotone Acceptance Ordered Upwind Method (MAOUM) \cite{AltonMitchell12}, a Dijkstra inspired method using static stencils, like the FM-LBR. The difference between these two methods is that the MAOUM stencil%
\footnote{%
The stencils for the MAOUM are here built using the {\rm ComputeUpdateSet} stencil construction routine described and used in all the numerical experiments of \cite{AltonMitchell12}. The paper \cite{AltonMitchell12} also outlines sufficient conditions (called $\delta$-NGA or DRB) for anisotropic stencils to be causal, but no explicit anisotropic stencil construction.%
}
$V(z)$ at a grid point $z\in \Omega \cap Z$, is isotropic, only depends on the anisotropy ratio $\kappa(\cM(z))$, and its boundary has cardinality $\cO(\kappa(\cM(z))^{\dim-1} )$; in contrast the FM-LBR stencil is anisotropic, aligned with the ellipse defined by $\cM(z)$, and of cardinality $\cO(1)$. 
The MAOUM stencils were precomputed and stored in a look-up table, resulting in a complexity $\cO(\kappa(\cM) N \ln N)$ for this algorithm in 2D. 


\begin{figure}
\begin{center}
\iftoggle{siam}{
\includegraphics[width=3cm]{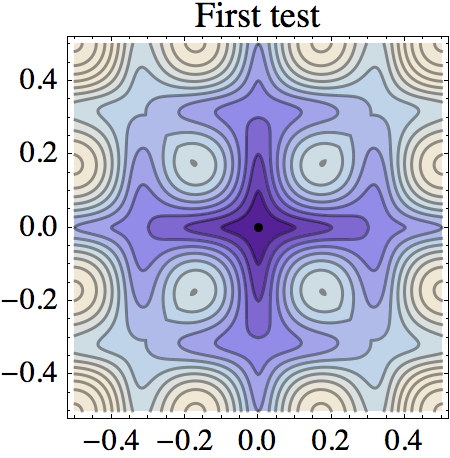}
\includegraphics[width=3cm]{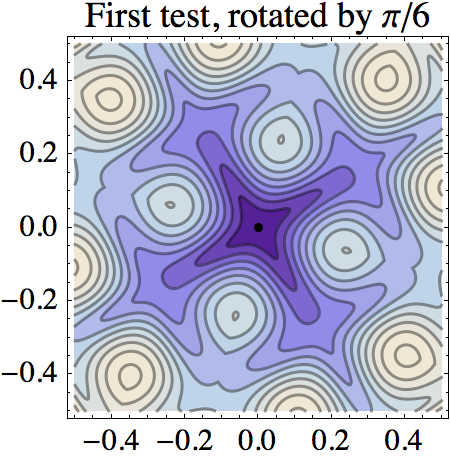}
\includegraphics[width=3cm]{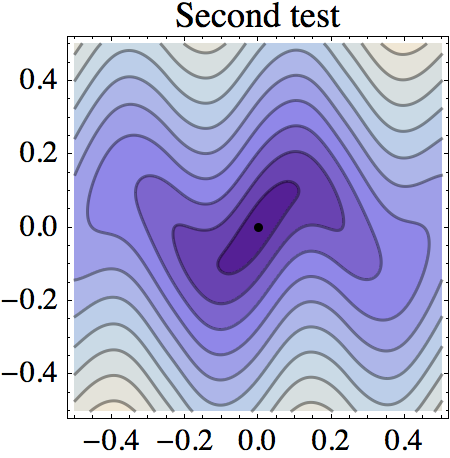}
\includegraphics[width=3cm]{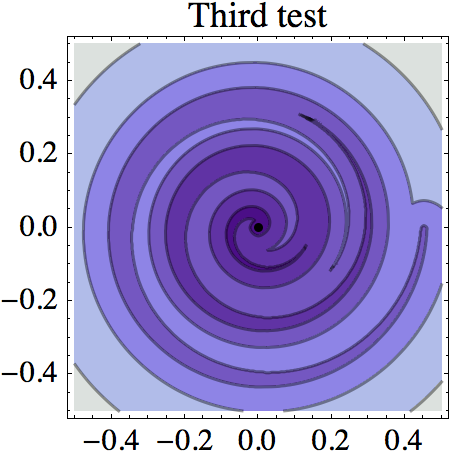}
}{
\includegraphics[width=3.8cm]{\pathPic FastMarchingIllus/VladBench/Vlad2_dist.png}
\includegraphics[width=3.8cm]{\pathPic FastMarchingIllus/VladBench/Vlad2Rot_dist.png}
\includegraphics[width=3.8cm]{\pathPic FastMarchingIllus/VladBench/Vlad_dist.png}
\includegraphics[width=3.8cm]{\pathPic FastMarchingIllus/VladBench/Snail2_contours.png}
}
\end{center}
\caption{
\label{fig:Vlad1}
Level lines of the solutions of the two dimensional test cases. 
}
\end{figure}

We consider application inspired test cases, which violate some of the simplifying assumptions used in our convergence analysis Proposition \ref{prop:Convergence}: they involve outflow boundary conditions, non-trivial Dirichlet boundary conditions in Appendix \ref{subsec:Accuracy}, and discontinuous Riemannian metrics in cases 3 and 4.
Their exact technical description is given in Remark \ref{rem:ExactDescription}.\\

The first test is a distance computation on a parametrized surface, considered in \cite{VladMSSP}. As shown on Figure \ref{fig:tables}, the FM-LBR is the fastest in terms of CPU time%
\footnote{%
All timings obtained on a 2.4Ghz Core 2 Duo, using a single core. Timings of the FM-LBR include the stencil construction, which typically accounts for $25\%$.
}, but is less accurate than the AGSI or the FM-8.
Rotating this test case by the angle $\theta = \pi/6$, and conducting the same experiment, shows a different story: the numerical errors of the AGSI and the FM-8 increase by a factor larger than $5$, while the FM-LBR, unaffected, is now the most accurate method, see Figure \ref{fig:tables}. 
The FM-LBR cuts $L^\infty$ and $L^1$ numerical errors by 40\% in comparison with the AGSI and the MAOUM, and CPU time by 85\%, while the FM-8 produces even larger errors.

Figure \ref{fig:tables} shows that the FM-LBR offers the best accuracy for more grid orientations $\theta$ than its alternatives. The maximal error and averaged error with respect to $\theta$ are also in favor of the FM-LBR.
The strong dependence of the AGSI and the FM-8 accuracy on the test case orientation is puzzling, and contrasts with the more consistent behavior of the MAOUM and the FM-LBR. 
The author's heuristic and personal interpretation of this phenomenon, which is open to debate and put into question by a reviewer, is that this test case is for $\theta=0$ dominated by (close to) axis-aligned anisotropy. The AGSI and the FM-8, which are based on small and fixed stencils, benefit from this special configuration; the FM-8 also works well for $\theta=\pi/4$, because its stencil includes the four diagonals.\\

\begin{figure}
\begin{center}
\iftoggle{siam}{
\begin{minipage}{9cm}
\begin{tabular}{c|cccc}
 &  FM-LBR& FM-8 & AGSI & MAOUM\\
 \hline
 & \multicolumn{4}{c}{First test}\\
CPU time & \tr{0.19} & \tr{0.19} & 1.01 & 1.28 \\ 
$L^\infty$ error & 3.99 & \tr{1.47} & 1.62 & 8.80\\
$L^1$ error & 1.13 & 0.53 & \tr{0.51} & 2.33\\
\hline
 & \multicolumn{4}{c}{First test, rotated by $\pi/6$}\\
CPU time & \tr{0.20} & 0.21 & 1.44 & 1.31 \\ 
$L^\infty$ error & \tr{5.52} & 12.5 & 9.45 & 8.56\\
$L^1$ error & \tr{1.46} & 3.42 & 2.51 & 2.52 \\
\hline
 & \multicolumn{4}{c}{Second test}\\
CPU time & \tr{0.076} & 0.079 & 0.77 & 0.36 \\ 
$L^\infty$ error & \tr{2.90} & 3.03 & 3.67 & 7.66\\
$L^1$ error & \tr{1.03} & 1.30 & 1.40 & 2.3\\
\end{tabular}
\end{minipage}
\hspace{-0.5cm}
\begin{minipage}{3cm}
\includegraphics[width=4cm]{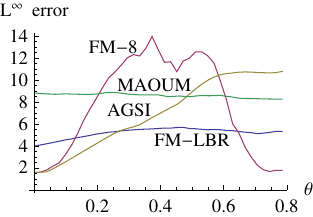}
\includegraphics[width=4cm]{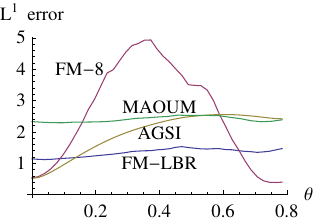}
\end{minipage}
}{
\begin{minipage}{9cm}
\begin{tabular}{c|cccc}
 &  FM-LBR& FM-8 & AGSI & MAOUM\\
 \hline
 & \multicolumn{4}{c}{First test}\\
CPU time & \tr{0.19} & \tr{0.19} & 1.01 & 1.28 \\ 
$L^\infty$ error & 3.99 & \tr{1.47} & 1.62 & 8.80\\
$L^1$ error & 1.13 & 0.53 & \tr{0.51} & 2.33\\
\hline
 & \multicolumn{4}{c}{First test, rotated by $\pi/6$}\\
CPU time & \tr{0.20} & 0.21 & 1.44 & 1.31 \\ 
$L^\infty$ error & \tr{5.52} & 12.5 & 9.45 & 8.56\\
$L^1$ error & \tr{1.46} & 3.42 & 2.51 & 2.52 \\
\hline
 & \multicolumn{4}{c}{Second test}\\
CPU time & \tr{0.076} & 0.079 & 0.77 & 0.36 \\ 
$L^\infty$ error & \tr{2.90} & 3.03 & 3.67 & 7.66\\
$L^1$ error & \tr{1.03} & 1.30 & 1.40 & 2.3\\
\end{tabular}
\end{minipage}
\begin{minipage}{4cm}
\includegraphics[width=5cm]{\pathPic FastMarchingIllus/VladBench/Vlad2LInf_MAOUM.pdf}
\includegraphics[width=5cm]{\pathPic FastMarchingIllus/VladBench/Vlad2L1_MAOUM.pdf}
\end{minipage}
}
\end{center}
\caption{
\label{fig:tables}
Tables of CPU time in seconds, $L^\infty$ error an averaged $L^1$ error (left). Accuracy, in the first test rotated by an angle $\theta\in [0, \pi/4]$ (this interval is enough, since the dependence in $\theta$ is $\pi/2$-periodic and even). In average over theta, CPU times are $0.21s$, $0.20s$, $1.37s$, $1.31s$, $L^\infty$ errors $5.16$, $7.64$, $6.86$, $8.57$ and averaged $L^1$ errors $1.34$, $2.58$, $1.95$, $2.40$ for the FM-LBR, FM-8, AGSI and MAOUM respectively. \emph{All errors are multiplied by $100$, for better readability.}}
\end{figure}

\begin{figure}
\begin{center}
\iftoggle{siam}{
\includegraphics[width=3cm]{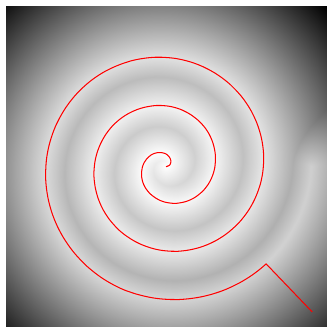}
\includegraphics[width=3cm]{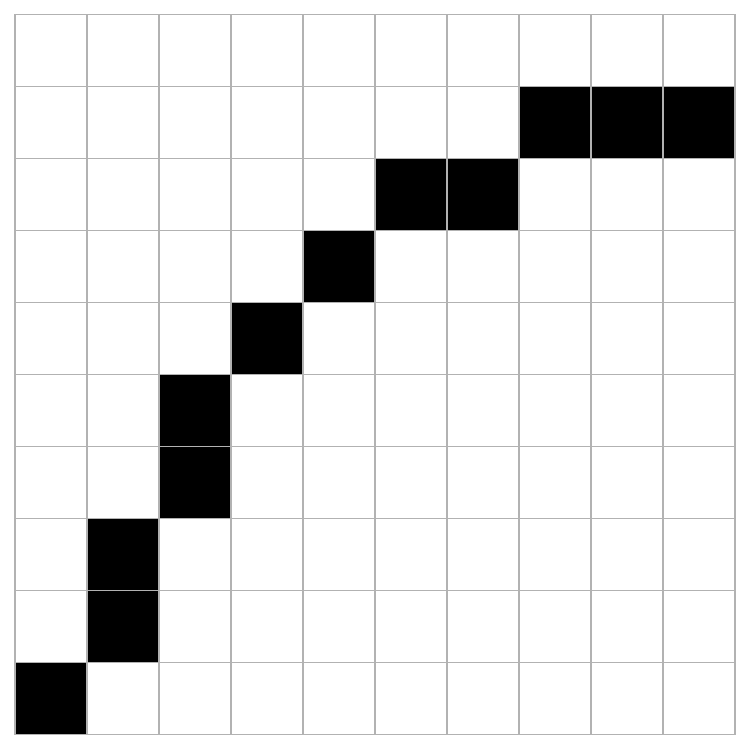}
\includegraphics[width=2.9cm]{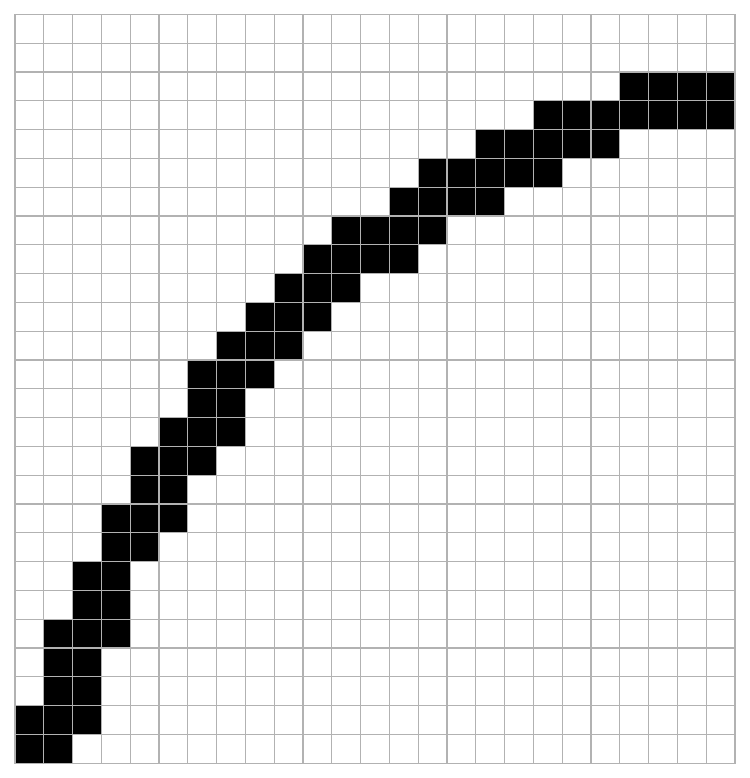}
\includegraphics[width=3cm]{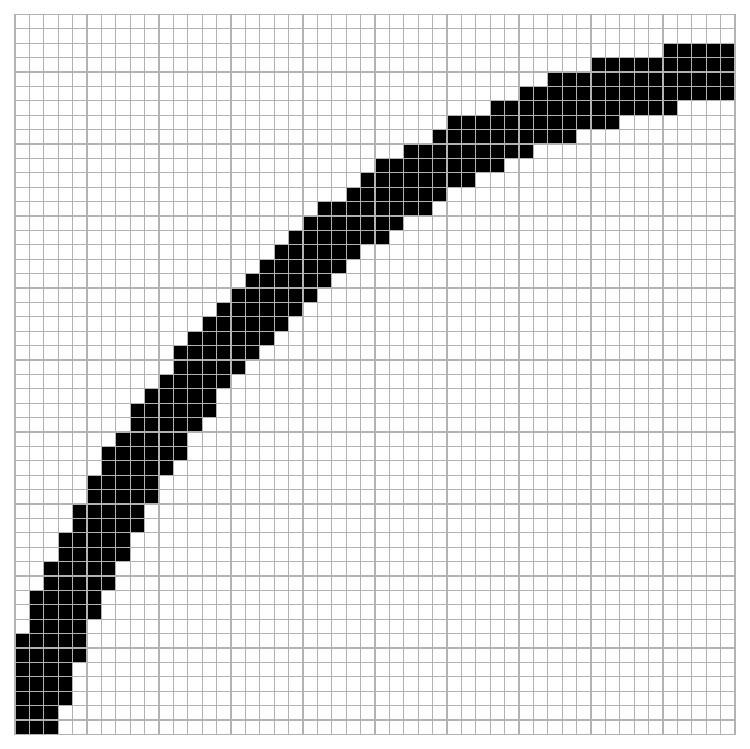}
}{
\includegraphics[width=3.8cm]{\pathPic FastMarchingIllus/VladBench/Snail2_New/snail2_unfolded_1000.pdf}
\includegraphics[width=3.8cm]{\pathPic FastMarchingIllus/VladBench/Snail2_New/Detail_201.pdf}
\includegraphics[width=3.7cm]{\pathPic FastMarchingIllus/VladBench/Snail2_New/Detail_501.pdf}
\includegraphics[width=3.8cm]{\pathPic FastMarchingIllus/VladBench/Snail2_New/Detail_1001.pdf}
}
\end{center}
\caption{
\label{fig:neigh}
Reference solution for the third test case (left). 
The Riemannian metric $\cM$ is anisotropic only on a thin band along a spiraling curve,  wide of a few grid points. Detail at resolutions $n\times n$, where $n$ equals $200$ (center left), $500$ (center right) and $1000$ (right).
}
\end{figure}

The second benchmark, discussed in \cite{VladThesis01,SV03}, is inspired by seismic imaging. 
There is 
no bias here towards axis-aligned anisotropy. 
As shown on the table Figure \ref{fig:tables}, the FM-LBR takes a smaller CPU time and offers a better accuracy than its alternatives.
Note that one can also construct configurations in which anisotropy is not axis-aligned \emph{and} the AGSI is more accurate than the FM-LBR. In some cases, the accuracy advantage of the AGSI even grows unboundedly as the anisotropy ratio of the metric tends to infinity. A heuristic analysis and prediction of this phenomenon is presented in Appendix \ref{subsec:Accuracy}, where it is illustrated with a fifth test case.\\  

The third \cite{BC10} and fourth test cases are relevant benchmarks if one's objective is to use fast marching methods for the segmentation of tubular structures, in medical images or volume data respectively.
The FM-LBR reveals its full potential, and stands out as the only practical option, in this more difficult setting which involves a discontinuous and highly anisotropic metric. 
The Riemannian metric tensor $\cM(z)$ is the identity matrix, except at points $z$ on the neighborhood of a curve $\Gamma$, where $\cM(z)$ has a small eigenvalue $\delta_0^2$ associated to an eigenvector tangent to $\Gamma$, and the other eigenvalue is $1$ on the orthogonal space. 
The shortest path joining the point $(1,-1)$ (resp.\ $(0,0,3)$) to the origin is extracted via ``gradient descent on the Riemannian manifold $(\Omega, \cM)$'', see Appendix \ref{sec:Paths} for details:
\be
\label{descent}
\gamma'(t) = -\cM(\gamma(t))^{-1} \nabla \distC(\gamma(t)).
\ee
By construction of the Riemannian metric $\cM$, traveling close and tangentially to the curve $\Gamma$ is cheap. This is reflected by the level lines of $\distC$, and by the allure of the minimal path, see Figures \ref{fig:Vlad1}, \ref{fig:neigh} and \ref{fig3d}.
Heuristically, this path joins the neighborhood of the curve $\Gamma$ in straight line, almost orthogonally, and then follows it.
The alignment of the minimal path with the direction of anisotropy, observed in this test case, is not an uncommon phenomenon.
The FM-LBR presumably benefits a lot from this behavior in terms of accuracy, since its stencils typically provide a good ``angular resolution'' in the direction of anisotropy, see Figures \ref{fig:BasisRotate}, \ref{fig:BasisScale}, \ref{fig:neigh}. Since in addition the stencil radii remain rather small for most anisotropy orientations, see \cite{M12b} for details, usually most updates for points in the ``fast band'' come from the fast band when the FM-LBR is run on these examples. When the fast band is missed however, accuracy degrades and zig-zag artifacts appear in the extracted path, see Figure \ref{fig:path}.

\begin{figure}
\begin{center}
\iftoggle{siam}{
\includegraphics[width=4cm]{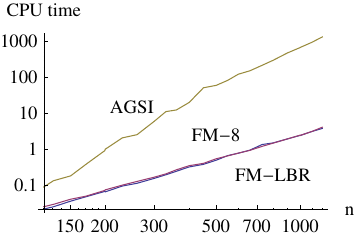}
\includegraphics[width=4cm]{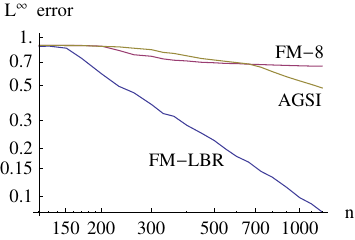} 
\includegraphics[width=4cm]{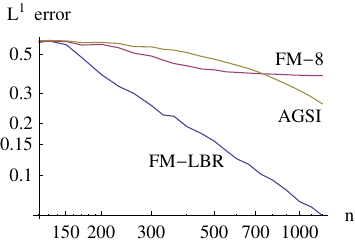}
}{
\includegraphics[width=5cm]{\pathPic FastMarchingIllus/VladBench/Snail2_New/CPUtimeFinsler.pdf}
\includegraphics[width=5cm]{\pathPic FastMarchingIllus/VladBench/Snail2_New/LInf.pdf} 
\includegraphics[width=5cm]{\pathPic FastMarchingIllus/VladBench/Snail2_New/L1.pdf}
}
\end{center}
\caption{
\label{fig2d}
CPU Time (left, in seconds), $L^\infty$ error (center), and averaged $L^1$ error (right) of the FM-LBR, FM-8 and AGSI, at several resolutions ranging from $120$ to $1200$ (log-log scale). 
}
\end{figure}

Third test case \cite{BC10}, in 2D. The different method's performance is illustrated on Figure \ref{fig2d}, except for the MAOUM which showed a poor accuracy, presumably due to the huge stencils it generated. 
The CPU time/resolution curve of the AGSI shows a stronger slope than for the one pass solvers, presumably reflecting its super-linear complexity $\cO(\mu(\cM) N^{\frac 3 2})$. 
The $L^\infty$ and $L^1$ error curves suggest that the FM-8 is not consistent in this test case, contrary to the FM-LBR and the AGSI. 
At the resolution $1000\times 1000$, typical in image analysis, the FM-LBR cuts the $L^\infty$ error by $80\%$ and the $L^1$ error by $75\%$ with respect to the AGSI, while reducing CPU time from 11 minutes to 2.5 seconds (!). As illustrated on Figure \ref{comp2d},  the better accuracy of the FM-LBR in this test case effectively translates into a better extraction of minimal paths. 
The reason for the, unrivaled, performance of the FM-LBR in this specific test case is partly elucidated in Appendix \ref{subsec:Accuracy}.\\

\begin{figure}
\begin{center}
\begin{tabular}{rccc}
Method $\sm$ Grid
& 200x200 & 500x500 & 1000x1000 \\
{\raise 12mm
\hbox{
FM-LBR
}}
&\includegraphics[width=3cm,height=3cm]{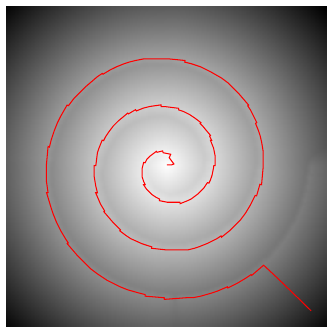}
&\includegraphics[width=3cm,height=3cm]{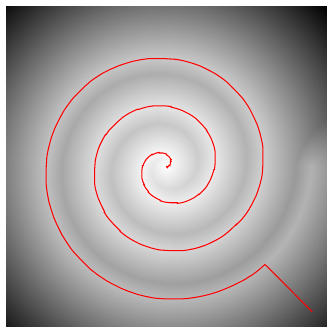}
&\includegraphics[width=3cm,height=3cm]{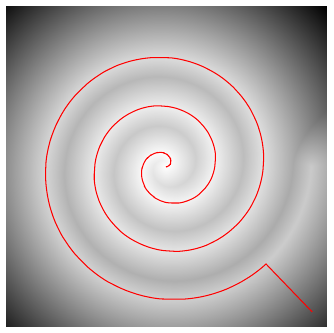}\\
{\raise 12mm
\hbox{
FM-8
}}
&\includegraphics[width=3cm,height=3cm]{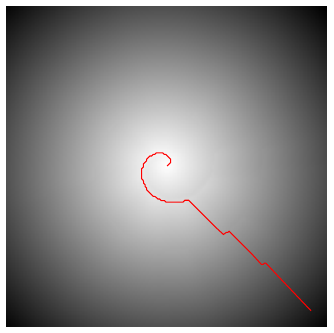}
&\includegraphics[width=3cm,height=3cm]{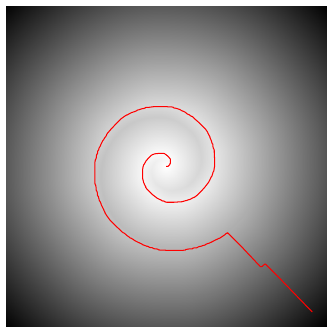}
&\includegraphics[width=3cm,height=3cm]{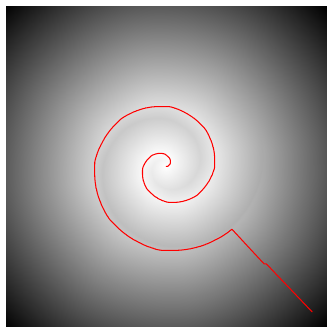}\\
{\raise 12mm
\hbox{
AGSI
}}
&\includegraphics[width=3cm,height=3cm]{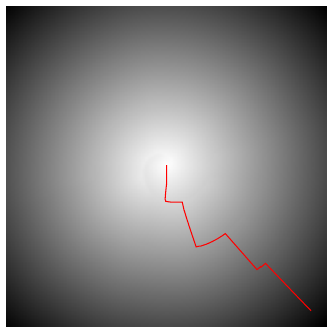}
&\includegraphics[width=3cm,height=3cm]{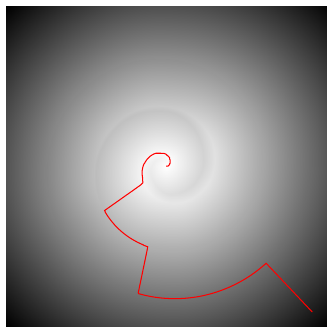}
&\includegraphics[width=3cm,height=3cm]{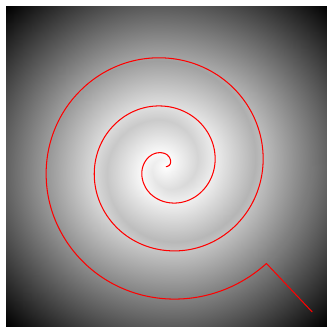}
\end{tabular}
\caption{Visual comparison of the accuracy of three algorithms, at three resolutions, in the 2D test case. Qualitatively, the approximate geodesic has the right behavior for a resolution as low as $170 \times 170$ with the FM-LBR, and $1000 \times 1000$ with the AGSI. 
This is presumably never the case for the FM-8, which is not consistent here. 
}
\label{comp2d}
\end{center}
\end{figure}

\begin{figure}
\begin{center}
\includegraphics[width=5cm]{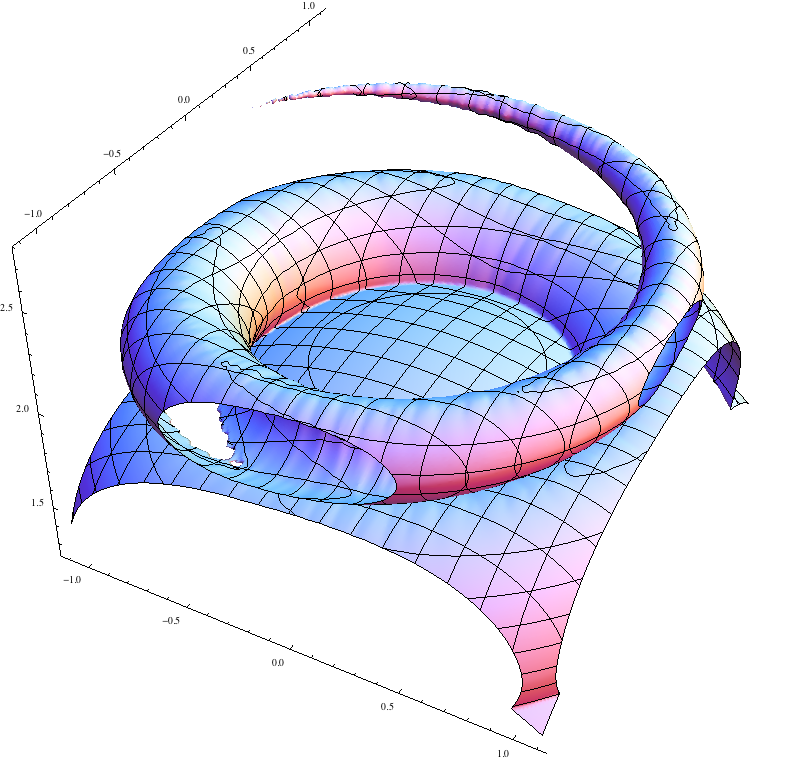}
{\raise 3mm \hbox{
\includegraphics[width=3cm]{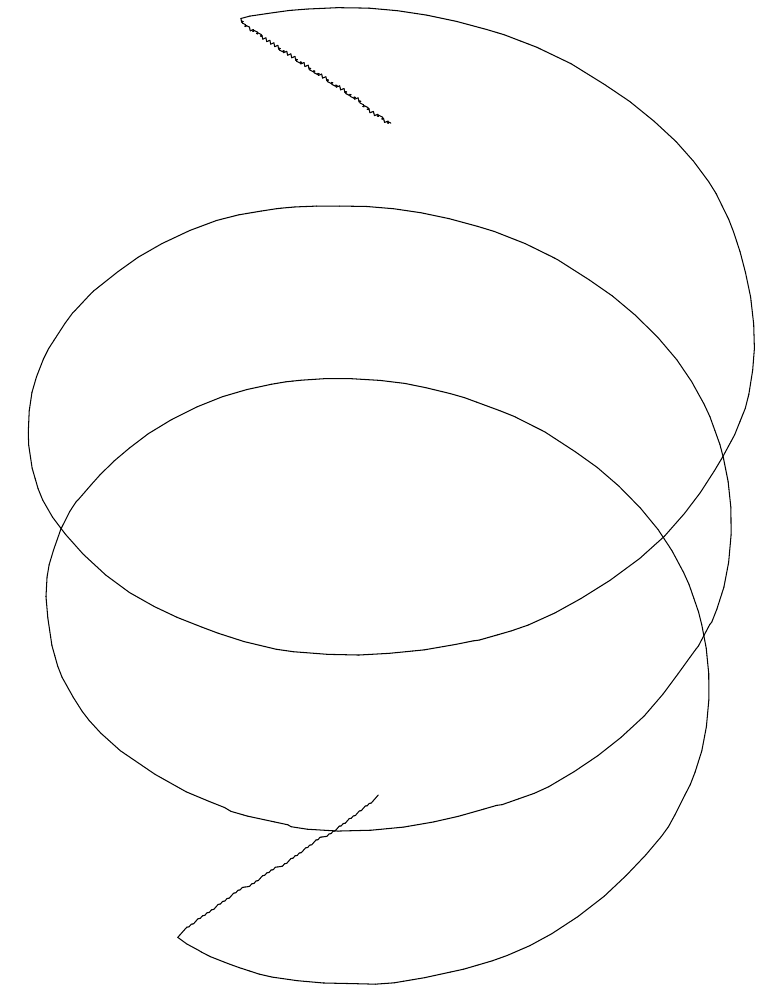}
}}
{\raise 10mm \hbox{
\includegraphics[width=4cm]{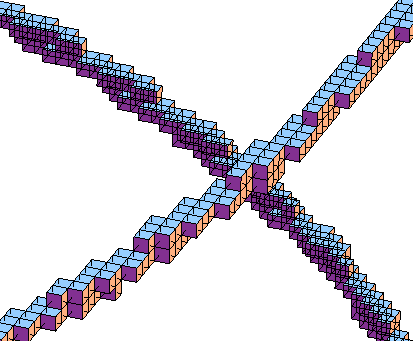}
}}
\end{center}
\caption{\label{fig3d} Results of the FM-LBR in the fourth, 3D, test case. Iso-surface $\{\dist(z)=2\}$ (left), and shortest path joining the points $(0,0,0)$ and $(3,0,0)$ (center). Detail of the discrete points (represented by small cubes), in the neighborhood of the curve $\Gamma(t)=(\cos \omega_0 t, \sin \omega_0 t, t)$, for which the Riemannian metric is not euclidean (right).}
\end{figure}

Fourth test case, in 3D.
CPU time was 105s for the FM-LBR, while the AGSI took 480s and failed to recover the minimal path presented on Figure \ref{fig3d} (center) (a straight line joining the two endpoints was obtained instead). 
The FM-LBR is capable of addressing a large scale (more than $10$ millions grid points), strongly anisotropic ($\kappa(\cM)=50$) three dimensional shortest path problem, with a good accuracy and within reasonable CPU time on standard laptop computer.

\begin{remark} The four application inspired test cases. Outflow boundary conditions, except for $\distC(0)=0$; hence the solution is the Riemannian distance to the origin: $\distC(z) = \distC(z,0)$. 
Numerical errors with respect to a $4000 \times 4000$ reference solution, bi-linearly interpolated, obtained in cases 1,2, with the AGSI, and in case 3 with the FM-ASR \cite{M12c}.
\label{rem:ExactDescription}
\begin{enumerate}
\item (Geometry processing \cite{VladThesis01}, $\kappa(\cM)\simeq 5.1$)
Compute the Riemannian distance from the origin $(0,0,0)$ on parametric surface of height map $z(x,y) := (3/4) \allowbreak \sin (3 \pi x) \sin (3 \pi y)$. Riemannian metric: $\cM(x,y) = \Id + \nabla z(x,y) \nabla z(x,y)^\trans$.
Coordinates $(x,y)$ restricted to the square $[-0.5,0.5]^2$, discretized on a $292\times 292$ grid, or in a second step to this square rotated by the indicated angle $\theta$.
\item (Seismic imaging \cite{VladThesis01,SV03}, $\kappa(\cM)=4$) 
The metric $\cM(x,y)$ has the two eigenvalues $0.8^{-2}$, $0.2^{-2}$, the former associated to the eigenvector $(1, (\pi/2) \cos (4 \pi x))$. 
Domain $[0.5,0.5]^2$, discretized on a $193 \times 193$ grid. 
\item (Tubular segmentation \cite{BC10}, $\kappa(\cM)=100$)
Define the curve $\Gamma(t) = t (\cos \omega_0 t, \sin \omega_0 t)$, $t \in [0,1]$. Set $\cM(z) = \Id$, except if there exists $0\leq t \leq 1$ and $0 \leq r \leq r_0$ such that $z = \Gamma(t)+r (\cos \omega_0 t, \, \sin \omega_0 t)$. In that case $\cM(z)$ has the eigenvalues $\delta_0^2$ and $1$, the former with eigenvector $\Gamma'(t)$. Parameters: $\omega_0 := 6 \pi$, $r_0 := \delta_0:= 0.01$. Domain: $[-1,1]^2$, grid sizes $n \times n$ with $120 \leq n \leq 1200$. 
\item (Tubular segmentation, $\kappa(\cM)=50$, 3D)
Define the curve 
$\Gamma(t) = (\cos \omega_0 t, \sin \omega_0 t, t)$, with $\omega_0 := (5/2) \pi$. Set $\cM(z)=\Id$, except if there exists $t, \lambda, \mu \in \R$ such that $z=\Gamma(t)+(\lambda \cos\omega_0 t, \lambda \sin \omega_0 t, \mu)$ and $\lambda^2+ \mu^2 \leq (r_0/2)^2$.  In that case $\cM(z)$ has the eigenvalues $\delta_0^2$ and $1$, the former with eigenvector $\Gamma'(t)$ and the latter with multiplicity $2$. Parameters: $\omega_0 := (5/2) \pi$, $\delta_0 = r_0 = 0.02$. Domain: $[-1.1,1.1]^2 \times [0,3]$, grid size  $200\times 200 \times 272$. 
\end{enumerate}
\end{remark}

%

\section*{Conclusion}

The FM-LBR, introduced in this paper, combines the Fast Marching algorithm with a concept from discrete geometry named Lattice Basis Reduction. It has the following strongpoints.
(I, Convergence) 
The FM-LBR is consistent for the anisotropic eikonal equation associated to any continuous Riemannian metric, of arbitrary anisotropy. 
(II, Complexity) It has a numerical cost comparable to classical isotropic Fast Marching, 
 independently of the problem anisotropy. 
(III, Accuracy) The accuracy of the FM-LBR is competitive in general, and striking in test cases, related to tubular segmentation in medical images, where the Riemannian metric has a pronounced anisotropy close to and tangentially to a curve.

These qualities come at the price of the specialization of the FM-LBR: 
(i) the Riemannian metric may not be replaced with a more general Finsler metric, see \cite{M12c} for an adaptation to this setting in 2D, (ii) the domain needs to be discretized on a cartesian grid, and (iii) of dimension $2$ or $3$. Hopefully these requirements are met in many applications, and future work will be devoted to the application of the proposed algorithm in the context of medical image processing.

\paragraph{Acknowledgement.}
The author thanks Pr A. Vladimirsky for constructive discussions on the choice of test cases and the FM-LBR accuracy, and Pr P. Q. Nguyen for pointing out the notion of obtuse superbase of a lattice.

\appendix

\section{Robust extraction of minimal paths}
\label{sec:Paths}

Obtaining the shortest path joining two given points is essential in motion planning control problems \cite{AltonMitchell12}, as well as in the envisioned application to tubular structure centerline extraction \cite{BC10}.
This involves solving the Ordinary Differential Equation (ODE) \eqref{descent}, a task less trivial than it seems. 
The author is conscious that a comparison of minimal paths, as on Figure \ref{comp2d}, reflects the properties of the ODE solver (and the time spent adjusting its sometimes numerous parameters), as much as those of the eikonal solver. This is done nevertheless due to the importance of minimal paths in applications.
Eikonal solvers based on the discrete fixed point problem \eqref{discreteSys}, such as the FM-LBR, FM-8 and AGSI, provide at each grid point $x \in \Omega \cap Z$ an estimate $\dist(x)$ of the distance $\distC(x)$, and in addition an estimate $v(x)$ of the direction and orientation of the distorted negative gradient $-\cM(x)^{-1} \nabla \distC (x)$, of the form: 
\be
\label{defvz}
v(x) := \sum_{1 \leq j \leq k} \alpha_j (z_j-x), 
\ee
where the integer $1 \leq k \leq \dim$, the \emph{positive} coefficients $(\alpha_j)_{j=1}^k$ and the vertices $(z_j)_{j=1}^k$ of $\partial V(z)$ are the barycentric coordinates of the point $y \in \partial V(x)$ achieving the minimum in the Hopf-Lax update operator \eqref{def:HopfLax}, in the face of $\partial V(x)$ containing $y$ and of minimal dimension. 

\begin{figure}
\begin{center}
\iftoggle{siam}{
\includegraphics[width=3cm]{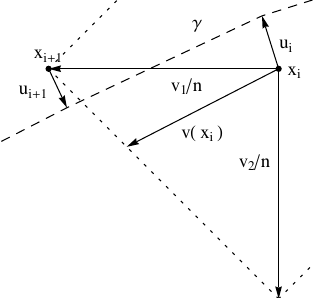}
\includegraphics[width=3cm]{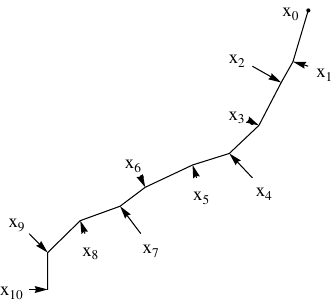}
\includegraphics[width=3cm]{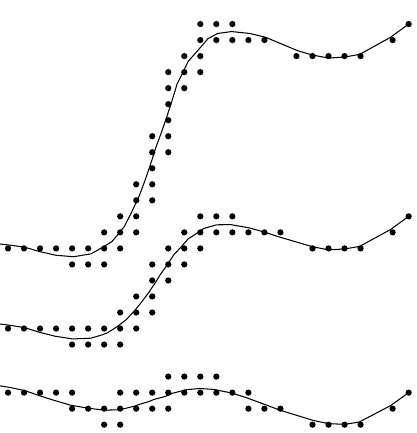}
\includegraphics[width=3cm]{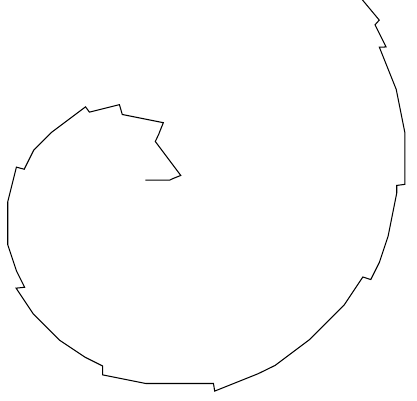}
}{
\includegraphics[width=3.8cm]{\pathPic FastMarchingIllus/Geodesics/GeoNotations.pdf}
\includegraphics[width=3.8cm]{\pathPic FastMarchingIllus/Geodesics/pathPointArrows.pdf}
\includegraphics[width=3.8cm]{\pathPic FastMarchingIllus/Geodesics/GeoVladEasy.pdf}
\includegraphics[width=3.8cm]{\pathPic FastMarchingIllus/Geodesics/GeoZigZag.pdf}
}
\end{center}
\caption{
\label{fig:path}
Left: Notations for the minimal path computation; the contour of the stencil $V(x_i)$ is shown dotted.
Center left: Grid points $x_0,\cdots,x_r \in \Omega \cap Z$,  corrections $u_0, \cdots, u_r \in \R^\dim$ shown as arrows, and piecewise linear path $\gamma$.
Center right: Grid Points $(x_i)_{i=1}^r$ and piecewise linear path $\gamma$ in the second test case at resolution $100\times 100$, 
using the FM-LBR. 
Right: In hard test cases, combining strong anisotropy, metric discontinuity, and low grid resolution, the extracted path exhibits zig-zag artifacts (detail of the third test case at resolution $n=200$).
}
\end{figure}

From this point, a typical approach to solve \eqref{descent} is to extend the values of $\dist$ or $v$ to the continuous domain $\Omega$ via an interpolation procedure, and then to use a black box ODE solver or a Runge Kutta method. 
Note that the accuracy usually expected from these high order methods is mostly doomed, since the discretization \eqref{discreteSys} of the eikonal equation is only first order, and since the vector field $\cM^{-1}\nabla \distC$ is discontinuous both at ``caustics'' and discontinuities of $\cM$.
A more significant issue is that computations frequently get stuck, despite the use of state of the art and/or commercial interpolation methods and ODE solvers, see e.g.\ \cite{AltonMitchell12} Figure 5.10.\\

We propose a method for the computation of minimal paths, which trades high order accuracy for robustness, and never gets stuck if the eikonal solver is Dijkstra inspired.  It takes advantage of the specific form \eqref{discreteSys} of the discretization of the eikonal equation, and does not rely on black box routines. It is also parameter free: there is not interpolation order or gradient step to adjust. 

\begin{algorithm}
\caption{
Minimal path computation, starting from a given grid point $x_0\in \Omega \cap Z$.
}
\label{algo:Path}
{\bf Initialisation:} $x \leftarrow x_0$, $u \leftarrow 0$ (mutable variables).\\
{\bf While} $x$ is not an initial source point, {\bf do}\\ 
 \begin{tabular}{cl}
 &Denote by $z_1, \cdots, z_k$ the grid points appearing in the expression \eqref{defvz} of $v(x)$.\\
 &Find $\lambda \in \R_+$, $1 \leq j \leq k$, which minimize $\|x+ u + \lambda v(x) - z_j \|$. \\ 
 &Perform updates $u \leftarrow x+u + \lambda v(x) - z_j$ and  $x \leftarrow z_j$.\\
 \end{tabular}
\end{algorithm}

The successive iterations of Algorithm \ref{algo:Path} generate grid points $x_0, \cdots, x_r \in Z$, and small correcting offsets $u_0, \cdots, u_r \in \R^m$.
If the causality property holds, see  Proposition \ref{prop:Causality}, then the values $(\dist(x_i))_{i=0}^r$ are strictly decreasing, so that the algorithm is guaranteed to terminate.
The piecewise linear path $\gamma : [0,r] \to \Omega$, parametrized so that $\gamma(i) := x_i+u_i$, satisfies the differential equation
\be
\label{paramGamma}
\gamma'(t) =  \lambda_{\lfloor t \rfloor} v(x_{\lfloor t \rfloor}), 
\ee
for non-integer $t \in [0,r]$, where the constants $(\lambda_i)_{i=0}^{r-1}$ are the minimizers in the second step of the while loop, and the vector $v(z)$, $z\in \Omega \cap Z$, is defined in \eqref{defvz}. The particularity of our path extraction method is that the direction field $v$ is not evaluated on the curve $\gamma$, but at the points $(x_i)_{i=0}^r$ which remain close as shown by the next proposition (the involved exponent 3 seems to be either over-estimated or a rare worst case scenario, in view of the experiments Figure \ref{fig:path}). 

\begin{proposition}
Let $\Omega \subset \R^2$  be an open bounded set, equipped with a Riemannian metric $\cM \in C^0(\overline \Omega, S_2^+)$, and an admissible family $(\cT(z))_{z \in \Omega}$ of meshes with ``Boundedness'' constant $R$. (See Definition \ref{def:Admissibility}, $R \lesssim \kappa(\cM)$ for the FM-LBR). Consider a cartesian grid $Z$ of scale $h$, equipped with the corresponding stencils, see \eqref{def:Grid} and \eqref{eq:VFromT}, and solve the discrete system \eqref{discreteSys}.
Let $\gamma \in C^0([0,r], \Omega)$ be a path extracted with Algorithm \ref{algo:Path} and parametrized as in \eqref{paramGamma}. Then for all $t\in [0,r]$
\be
\label{upperGammaX}
\| \gamma(t) - x_{\lfloor t \rfloor} \| \leq C h R^3,
\ee
where $C$ is an absolute constant (i.e.\ independent of $\Omega$, $(\cT(z))_{z \in \Omega}$, $\cM$, or the path origin).
\end{proposition}

The rest of this appendix is devoted to the proof.
Consider a fixed $0 \leq i < r$, and observe that for all $t\in [i,i+1[$ one has, since $\gamma$ is linear on this interval and since $\gamma(i) = x_i+u_i$:
\begin{equation}
\label{eq:gammaDist}
\| \gamma(t) - x_{\lfloor t \rfloor} \| \leq \max \{ \| \gamma(i+1) - x_i\| ,\ \|\gamma(i) - x_i\| \} 
\leq \max \{ \| u_{i+1}\|  + \|x_{i+1} - x_i\|,\ \|u_i\|\}.
\end{equation}
In order to avoid notational clutter, we denote $x:=x_i$, $x':=x_{i+1}$, $u:=u_i$, $u':=u_{i+1}$ and $v := v(x_i)$.
Let $k \in \{1,2\}$ and let $z_j$, $1 \leq j \leq k$, be the vertices of $\partial V(x)$ appearing in the expression \eqref{defvz} of the discrete negative gradient $v$. 

We assume without loss of generality that the grid is $Z := \Z^2$, so that $v_j := z_j-x$ is a vertex of  $\partial \cT(x)$, for all $1 \leq j \leq k$. Hence $\|v_j\| \leq R$ by (Boundedness) and if $k =2$ then $\det(v_1, v_2)$ is a non-zero integer. In particular $\|x'-x\| = \|v_j\| \leq R$, for some $1 \leq j \leq k$.
By construction (second step of the while loop in the path computation), we have for any $\lambda \in \R_+$ and any $1 \leq j \leq k$
\be
\label{inequp}
\| u' \| \leq \| u + \lambda v - v_j\|.
\ee
We prove in the following an upper bound of the form $\|u'\| \leq \max \{\|u\|, C R^3\}$, which by an immediate induction argument implies $\|u_i\| \leq C R^3$ for any $1 \leq i \leq r$. Using \eqref{upperGammaX}, our previous estimate on $\|x_{i+1}-x_i\|$, and rescaling by a factor $h$, we obtain as announced \eqref{upperGammaX}.
Case $k=1$: choosing $\lambda=1$, $j=1$, and observing that $v = v_1$, we obtain $\|u'\| \leq \|u + 1 \times v - v_1\|= \|u\|$. The second case begins with a lemma.

\begin{lemma}
Let $v_1,v_2 \in \R^2$, and let $v := \alpha_1 v_1+\alpha_2 v_2$, with $\alpha_1, \alpha_2>0$.
Let $\mu := \sqrt 2$, and let $w_1 := v_1 \mu -  v_2/\mu$, $w_2 :=  v_2 \mu -  v_1 /\mu$.
Then one can choose $\lambda \in \R_+$ and $1 \leq j \leq 2$, such that $v_j - \lambda v$ is positively proportional to any of the following vectors: $w_1$, $w_2$, $-w_1$ or $-w_2$, with a proportionality constant $0 < \nu \leq \mu$.
\end{lemma}
\begin{proof}
In the case of $w_1$, choose $j:=1$, $\lambda := 1/(\alpha_1+ \mu^2 \alpha_2)$, so that $\nu = 1/(\mu + \mu^{-1}\alpha_1/ \alpha_2 ) \leq 1/\mu$.
In the case of $-w_1$, choose $j:=2$, $\lambda := 1/(\alpha_2+\mu^{-2} \alpha_1)$, so that $\nu =1/(\mu^{-1}+\mu \alpha_2/\alpha_1) \leq \mu$.
The cases of $w_2$ and $-w_2$ are similar. 
\end{proof}

Case $k=2$. We use the notations of the lemma. Denoting by $A$ the matrix of lines $w_1,w_2$, there exists 
$1 \leq k \leq 2$ and $\ve \in \{-1,1\}$ such that 
\begin{equation}
\label{uwLower}
 \sqrt 2 \<u, \ve w_k\>  \geq \sqrt{ \<u, w_1\>^2 + \<u, w_2\>^2 } = \|A u\|  \geq \|A^{-1}\|^{-1} \|u\|.
\end{equation}
The norm $\|A^{-1}\|^{-1}$ is estimated as follows. $|\det A| = (\mu^2-\mu^{-2}) |\det (v_1,v_2)| \geq \mu^2-\mu^{-2}$, since  $\det(v_1,v_2)$ is a non-zero integer. On the other hand $\|w_j\| \leq (\mu+\mu^{-1}) \max\{\|v_1\|,\|v_2\|\} \leq C_0 R$ with $C_0=2 (\mu+\mu^{-1})$, hence $\|A\| \leq \sqrt{\|w_1\|^2 +\|w_2\|^2}\leq \sqrt 2 C_0 R$. Finally $\|A^{-1}\|^{-1} = |\det A| / \|A\| \geq C_1 /R$ where $C_1 := (\mu-\mu^{-1})/(\sqrt 2 C_0)$. 

We next choose $\lambda$ and $j$, using the previous lemma, so that $v_j - \lambda v = \nu \ve  w_k$, with $\ve$, $k$ as in \eqref{uwLower}, and where $0 < \nu \leq \mu := \sqrt 2$.
Hence, using \eqref{inequp}:
\begin{align*}
\| u' \|^2 &\leq \| u - (v_j - \lambda v)\|  = \|u - \nu \ve w_k\|^2 = \|u\|^2 - 2 \nu \<u, \ve w_k\>+ \nu^2 \|w_k\|^2 \\
& \leq \|u\|^2 - 2 \nu \|A^{-1}\|^{-1} \|u\| + \nu^2 \|w_k\|^2 \leq \|u\|^2 - 2 C_1 \|u\|/R+ C_0^2 R^2.
\end{align*}

If $\|u\| \geq C_2 R^3$, with $C_2 := C_0^2/(2C_1)$, then $\|u'\| \leq \|u\|$. If $\|u\|$ is below this bound, then choosing $\lambda = 0$  in \eqref{inequp} yields $\|u'\| \leq \|u\|+\|v_1\| \leq C_2 R^3 + 2 R$. Thus $\|u'\| \leq \max \{\|u\|, C R^3\}$ as announced, which concludes the proof.

\section{Stencil size, discretization errors and metric regularity}
\label{subsec:Accuracy}

Experience in the discretization of Partial Differential Equations (PDEs) tells that robust and accurate numerical schemes are usually based on small, localized and isotropic stencils. The FM-LBR, which achieves causality in the eikonal equation by the use of long range, sparse and highly anisotropic stencils, seems to violate this principle. We give in this section a \emph{heuristic} analysis of its accuracy, which explains its excellent performance in the third and fourth tests (once a source of puzzlement for the author and the reviewers), but also exposes some weaknesses.

Let us emphasize that the computational domain $\Omega$ is equipped with two geometries.
\begin{itemize}
\item The \emph{extrinsic} Euclidean geometry, inherited from the embedding $\Omega \subset \R^\dim$. 
\item The \emph{intrinsic} Riemannian geometry, given by the Riemannian metric $\cM$ on $\Omega$, which is part of eikonal PDE structure.
\end{itemize}
The AGSI stencils are small, localized, with respect to the extrinsic Euclidean distance. The FM-LBR stencil at a point $z \in \Omega$ is built from an $\cM(z)$-reduced basis, see \S \ref{sec:MeshProperties}, which consists of the smallest linearly independent vectors of $\Z^\dim$ in the local norm $\|\cdot\|_{\cM(z)}$. Hence this stencil is by construction small and localized \emph{in the twisted perspective} of the intrinsic Riemannian distance. We refer to \cite{M12b} for a quantitative estimate of the size of the FM-LBR stencil, from this perspective, in average over all orientations of the discretization grid.
To better reflect the shapes of these stencils, which (as much as possible) adapt their orientation and aspect ratio to the metric $\cM$, but keep a constant volume 
(they are built of $(m+1)!$ simplices of volume $h^m/m!$), we introduce a normalized metric $\widehat \cM$: for all $z\in \Omega$
\begin{equation}
\label{normalizedM}
\widehat \cM(z) := \det(\cM(z))^{-\frac 1 \dim} \cM(z),
\end{equation}
so that $\det(\widehat \cM(z)) = 1$ identically. For all $x,y\in \overline \Omega$ we denote by $\distC(x,y)$ (resp.\ $\widehat \distC(x,y)$) the Riemannian distance \eqref{def:Length} on $\Omega$ associated to $\cM$ (resp.\ $\widehat \cM$).

Bellman's optimality principle applied to the solution $\distC$ of the eikonal equation \eqref{eikonal}, reads:  for all $x \in V \subset \Omega$
\begin{equation*}
\distC(x) = \min_{y \in \partial V} \distC(x,y) + \distC (y).
\end{equation*}
The Hopf-Lax update operator \eqref{def:HopfLax} reflects this identity at discrete level, up to two approximations:
\begin{align}
\label{approxM}
\distC(x,y) & \approx  \|x-y\|_{\cM(x)},\\ 
\label{approxD}
\distC (y) &\approx \interp_{V(x)} \dist (y),
\end{align}
where $V(x)$ denotes the stencil at $x$ (which is given under the form of a triangulation of a neighborhood of $x$), $\interp_V$ the linear interpolation operator on $V$, and $y \in \partial V$. 

If one ignores the issue of the scheme causality, then the choice of stencil should be dictated by the local regularity properties of the quantities that are approximated on it. 
Discretization \eqref{approxM} amounts to approximating the metric $\cM$ with the constant $\cM(x)$ on the stencil $V(x)$. Such piecewise constant approximation errors are controlled by Lipschitz regularity constants. The natural distance on $S_\dim^+$ is defined in \eqref{def:dtimes}, but the distance on $\Omega$ should be chosen appropriately so as to reflect the geometry of the computational stencils. Indeed, the AGSI (resp.\ FM-LBR) stencil at a point $x\in \Omega$ is heuristically not much different from ball centered at $x$ and of radius the grid scale $h$, defined with respect to the euclidian distance (resp.\ Riemannian distance $\widehat D$). 
The AGSI stencil should therefore be preferred if the metric $\cM$ has a small Lipschitz regularity constant $K_0$ with respect to the extrinsic Euclidean distance:
\begin{equation}
\label{regME}
\dtimes(\cM(x), \cM(y)) \leq K_0 \|x-y\|.
\end{equation}
On the other hand, the FM-LBR stencil is more suitable for metrics which have a small Lipschitz regularity constant $K_1$ with respect to the intrinsic (up to the normalization \eqref{normalizedM}) distance $\widehat D$:
\begin{equation}
\label{regMD}
\dtimes(\cM(x), \cM(y)) \leq K_1 \widehat D(x,y).
\end{equation}
In the fifth numerical example below, the most accurate of these two methods can indeed be guessed from the ratio $K_0/K_1$.
Regularity conditions of the form \eqref{regMD} arise naturally in the study of anisotropic mesh generation, see Part III of \cite{Mirebeau10}. 
The Riemannian metric involved in the third and fourth numerical tests of this section, inspired by applications to tubular structure segmentation \cite{BC10}, varies slowly in the direction of the eigenvector associated to the small eigenvalue of $\cM(z)$ (the direction of the tube), but quickly (in fact discontinuously) in the orthogonal direction. Thus, although discontinuous, it is heuristically not far from satisfying \eqref{regMD}, which explains the exceptional performance of the FM-LBR on these specific examples.

The interpolation error \eqref{approxD} is harder to estimate, yet in favor of the FM-LBR stencil let us mention the intrinsic $1$-Lipschitz regularity of the solution: for all $x,y \in \Omega$
\begin{equation*}
|D(x) - D(y) | \leq D(x,y).
\end{equation*}
In the special case of a constant metric, where \eqref{approxM} is exact and all error comes from \eqref{approxD}, the FM-LBR stencil offers the best accuracy, see \cite{M12b}. \\
%

The following illustrative example was proposed by \Vladimirsky : let $M \in S_\dim^+$, let $u \in \R^\dim$,  and let $\cM : \R^\dim \to S_\dim^+$ be the Riemannian metric defined by 
\begin{equation}
\label{def:MExp}
\cM(z) := M \exp( \<u, z\>).
\end{equation}
Assume for normalization that $\det (M)=1$, so that $\widehat \cM = M$ identically, and $\widehat D(x,y) = \|x-y\|_{M}$ for all $x,y \in \R^\dim$. Then, with $D := M^{-1}$:
\begin{align*}
\dtimes(\cM(x), \cM(y)) = |\<u, x-y\>| &\leq  \|u\| \|x-y\|,\\
\dtimes(\cM(x), \cM(y)) = |\<u, x-y\>| &\leq \|u\|_{D} \|x-y\|_{M} = \|u\|_{D} \widehat D(x,y).
\end{align*} 
The Lipschitz regularity constants are therefore $K_0 = \|u\|$, and $K_1 = \|u\|_{D}$. The discretization \eqref{approxM} is hence likely more accuracte with the AGSI stencil if $\|u\| \leq \|u\|_D$, and with the FM-LBR stencil otherwise. Defining for all $z\in \R^\dim$
\begin{equation}
\label{def:DExp}
\dist(z) := \exp(\<u,z\>)/\|u\|_D,
\end{equation}
we observe that this map is unimodular: $\|\nabla \dist (z) \|_{\cM(z)^{-1}} = 1$. The value $\dist(z)$ can also be regarded as the geodesic distance from $z$ to a point at infinity in the direction of $-D u$.
The characteristic curves of the solution are parallel straight lines, of direction $D u$. 
Although the present discussion is about accuracy rather than CPU time, let us mention that thanks to this special property, and as pointed out by A. Vladimirsky, the AGSI converges in a single pass in the numerical tests below (provided its priority queue is suitably initialized:  the upwind boundary points $z$ must be sorted by increasing values of their scalar product $\<z, D u \>$ with the characteristics direction $D u$). 
As a result, and in contrast with \S \ref{sec:num}, the AGSI CPU time is here only \emph{half} of the FM-LBR CPU time.

Our fifth and last numerical test, involves a metric depending on three parameters $\kappa\geq 1$, $\theta, \vp\in \R$: denoting $e_\theta := (\cos \theta, \sin \theta)$, and $(x,y)^\perp := (y,-x)$,
\begin{equation}
\label{parametrizedMetric}
\cM(z; \, \kappa, \theta, \vp) := M(\kappa, \theta) \exp(\<z, e_\vp\>), \ \text{ with } M(\kappa, \theta) := \kappa e_\theta e_\theta^\trans + \kappa^{-1} e_\theta^\perp (e_\theta^\perp)^\trans .
\end{equation}
For each matrix $M = M(\theta, \vp)$, given by its condition number $\kappa$, and anisotropy direction $\theta$, we consider different unit vectors $u = e_\vp$, given by their angle $\vp$ with respect to the $x$-axis. For $\vp = \theta$ the FM-LBR is favored, since $K_0/K_1$ takes its maximal value $\sqrt {\kappa}$. For $\vp = -\pi/4$, the AGSI is favored, since $K_0/K_1$ is close to its minimal value $1/\sqrt \kappa$, and the interpolation error \eqref{approxD} vanishes.
The domain is the $[-2,2]^2$ square discretized on a $100 \times 100$ grid. The FM-LBR stencil is included in a square of $(2 w+1) \times (2 w+1)$ pixels, where $w$ respectively equals $2$, $3$, and $5$ for the three different pairs $(\kappa,\theta)$ in Table \ref{table:ExpMet}. 
The boundary condition \eqref{def:DExp} is applied on a the upwind part of the square boundary, on a layer of width $w$. The numerical tests presented in Table \ref{table:ExpMet} are typical of the authors experience, and tend to agree with the above heuristical error analysis. Note however that the FM-LBR performance is unexpectedly bad%
\footnote{%
The FM-LBR in that case produces large errors close to the \emph{downwind} boundary, presumably due to its large stencils.  Removing a $5$ pixel band on the boundary yields in this case the $L^\infty$ errors: 32.3 (FM-LBR) and 45.8 (AGSI), in favor of the FM-LBR as predicted from the Lipschitz constants ratio. 
} in case $^*$, and unexpectdly good in cases $\dagger$.

The accuracy advantage of the AGSI is larger than a factor $6$ for the last set of parameters in Table \ref{table:ExpMet}, and it may of course grow unboundedly as the anisotropy ratio $\kappa$ tends to infinity, for suitable angles $\theta, \vp$. Anisotropy does therefore, sometimes, play against the FM-LBR accuracy.

\begin{table}

\begin{center}
\begin{tabular}{|c|c|c|c|c|c|c|c|}
\hline
\multirow{2}{*}{$\kappa$} & \multirow{2}{*}{$\theta$} & \multirow{2}{*}{$\vp$} & $K_0/K_1$ & 
\multicolumn{2}{c|}{$L^1$ error} & \multicolumn{2}{c|}{$L^\infty$ error}  \\ 
& & & $=\|u\|/\|u\|_D$   & FM-LBR & AGSI & FM-LBR & AGSI \\ 
\hline
\hline
\multirow{3}{*}{3} & \multirow{3}{*}{$\pi/3$} & 
        $\pi/3$  & 1.73 & \tr{2.78} & 6.96 & \tr{40.2} & 60.5 \\ 
 & & $\pi/6$ & 1 & 2.80 & 3.07 & 17.1 & 18.5 \\ 
 & & $- \pi/4$ & 0.59 & 3.95 & \tr{1.45} & 37.4 & \tr{13.7} \\ 
\hline
\hline
\multirow{3}{*}{10} & \multirow{3}{*}{$3 \pi/8$} & 
        $3 \pi/8$ & 3.16 & \tr{3.74} & 9.45 & \tr{38.4} & 79.5 \\ 
&&   1.48 & 1      & 1.34 & 2.03 & 7.44 & 11.4 \\ 
&&   $-\pi/4$ & 0.34 & 2.92 & \tr{0.83} & 27.3 & \tr{7.82} \\ 
\hline
\hline
\multirow{3}{*}{30} & \multirow{3}{*}{$\pi/8$} & 
        $\pi/8$ & 5.48 & \tr{3.89} & 6.62 & 94.3$^*$ & \tr{61.0} \\ 
    &&  0.57 & 1      & \tr{0.91}$^\dagger$ & 2.55 & \tr{4.62}$^\dagger$ & 12.9 \\ 
    && $-\pi/4$ & 0.20 & 3.24 & \tr{0.49} & 29.4 & \tr{4.51} \\ 
    \hline
\end{tabular}
\end{center}

\caption{Metric $z \mapsto \cM(z; \,\kappa, \theta, \vp)$, see \eqref{parametrizedMetric}, on the $[-2,2]^2$ square discretized on a $100\times 100$ grid. The most accurate algorithm, among the AGSI and the FM-LBR, can in most cases be predicted by the Lipschitz constants ratio $K_0/K_1$, at least when it is far from $1$. 
The star points out an exception. 
CPU time approximatively $10$ ms for the AGSI, and $20$ ms for the FM-LBR.
Numerical errors multiplied by 100 for better readability.}
\label{table:ExpMet}
\end{table}

\end{document}